\documentclass{article}
\usepackage[multiple]{footmisc}
\usepackage[utf8]{inputenc}
\usepackage[T1]{fontenc}
\usepackage[left=1.8cm,right=1.8cm,top=1.8cm,bottom=1.8cm]{geometry}
\usepackage{tikz}
\usepackage{enumerate,mathtools}
\usepackage{amsmath}
\usepackage{amssymb}
\usepackage{amsthm,dsfont}
\usepackage{enumitem}   
\usepackage{mathtools}
\mathtoolsset{showonlyrefs}
\newtheorem{thm}{Theorem}[section]
\newtheorem{lem}[thm]{Lemma}
\newtheorem{prop}[thm]{Proposition}
\newtheorem{cor}[thm]{Corollary}
\newtheorem{rem}[thm]{Remark}

\numberwithin{equation}{section}

\title{Variance Bounds: Some Old and Some New}

\author{Clément Deslandes\footnote{C.M.A.P. \'Ecole Polytechnique, Palaiseau, 91120, France \& Georgia Institute of Technology, Atlanta, GA, 30332, USA (\texttt{clement.deslandes@poytechnique.edu}).}\qquad Christian Houdré\footnote{School of Mathematics, Georgia Institute of Technology, Atlanta, GA, 30332, USA (\texttt{houdre@math.gatech.edu}).} 
	\footnote{Research supported in part by the awards $524678$ and 
		${\rm MP\!\!-\!\!TSM-}\!00002660$ from the Simons Foundation. 
		}}

\newcommand\ens[1]{\{1,\dots,#1\}}
\def\Var{{\textrm{Var}}\,}
\def\Cov{{\textrm{Cov}}\,}
\def\bbr{\mathbb R}
\def\esp {\mathds{E}}
\def\pr {\mathds{P}}
\begin{document}
\maketitle

\begin{abstract}
For functions of independent random variables, 
various upper and lower variance bounds are revisited in diverse settings. These are then specialized to the Bernoulli, Gaussian, infinitely divisible cases 
and to Banach space valued random variables. 
Frameworks and techniques vary from jackknives through semigroups and beyond. 
Some new applications are presented, in particular,  recovering and improving upon all the known estimates on the variance of the length of the longest common subsequences of two random words.
\end{abstract}
\section{Introduction and preliminary results}

We revisit below various lower and upper bounds on the variance of functions of independent random variables. 
Throughout and unless otherwise noted, $X_1,\dots, X_n,X'_1,\dots,X'_n$ are independent random variables such that for 
all $k\in\ens{n}$, $X_k$ and $X'_k$ are identically distributed, while $S:\bbr^n\to \bbr$ is a 
Borel function such that $\esp S(X_1, \dots, X_n)^2 < +\infty$.  
Next, and if $S$ is short for $S(X_1, \dots, X_n)$, for any $k\in\ens{n}$, let $S^k:= S(X_1, \dots,X_{k-1},X'_k,X_{k+1},\dots, X_n)$ and 
more generally if $\alpha\subset\ens{n}$, let $S^\alpha$ be defined as $S(X_1,\dots,X_n)$ but with $X_k$ replaced by $X'_k$ for all $k\in\alpha$, and so $S^\emptyset = S$, while $S^{\ens{n}}$ is an independent copy of $S$.  
With the help of these preliminary notations, we next recall the definitions of various quantities which will play an important role in the sequel.

Following \cite{lugosi}, for $k\in\ens{n}$, let 
\begin{equation}\label{defB} 
B_k:=\esp \frac{1}{n!}\sum_{i\in\mathfrak{S}_n}S(S^{i_1,\dots, i_{k-1}}-S^{i_1,\dots, i_{k}}),
\end{equation}
where $\mathfrak{S}_n$ is the symmetric group of degree $n$ and where for $k=1, S^{i_1,\dots, i_{k-1}}=S$. As the following sum is telescopic:
\begin{equation*}\label{telescop}
\sum_{k=1}^n B_k=\esp \frac{1}{n!}\sum_{i\in\mathfrak{S}_n}S(S-S^{i_1,\dots ,i_{n}})=\Var S.
\end{equation*}
One key fact motivating the definition of the $B_k$'s is that they can be rewritten as:
\begin{equation}
B_k=\esp \frac{1}{2n!}\sum_{i\in\mathfrak{S}_n}(S-S^{i_k})(S^{i_1,\dots ,i_{k-1}}-S^{i_1,\dots ,i_{k}}).
\end{equation}
Indeed, if $\alpha,\beta\subset\ens{n}$,
\begin{equation}\label{difsym}
\esp \left(S^\alpha S^\beta\right)=\esp \left(SS^{\alpha\Delta\beta}\right),
\end{equation} 
where as usual $\Delta$ denotes the symmetric difference operator, so 

\begin{equation}.\label{basecasefirst}
\esp (S-S^{i_k})(S^{i_1,\dots, i_{k-1}}-S^{i_1,\dots, i_{k}})=2\esp \left(SS^{i_1,\dots, i_{k-1}}-SS^{i_1,\dots, i_k}\right)
\end{equation}

Next, for all $k\in\ens{n}$, let $\Delta_k S:=S-S^k$, and for $k\neq {\ell}$, let $\Delta_{k,{\ell}}S:=\Delta_k(\Delta_{\ell} S)=S-S^k-S^{\ell}+S^{k,{\ell}}$ (note the commutativity property: $\Delta_k(\Delta_{\ell} S)=\Delta_{\ell}(\Delta_k S)$). Iterating this process, let $\Delta_{i_1,\dots,i_k}S:=\Delta_k(\Delta_{i_1,\dots,i_{k-1}}S)$. Using this notation, we have \begin{equation}\label{propB}
B_k=\esp \frac{1}{2n!}\sum_{i\in\mathfrak{S}_n}(\Delta_{i_k}S)(\Delta_{i_k}S)^{i_1,\dots,i_{k-1}},
\end{equation}
and so $B_k\geq 0$ since if $U,U'$ and $V$ are independent with $U$ and $U'$ identically distributed, then for any function $F$ such that $F(U,V)$ is integrable, $\mathbb{E}\left(F(U,V)F(U',V)\right)=\mathbb{E}\left(\mathbb{E}\left(F(U,V)|V\right)^2\right)\geq 0$. We are now ready to generalize the approach used to go from \eqref{defB} to \eqref{propB}, leading to novel properties of the $B_k's$.

\begin{lem} Let  $\alpha\ne \emptyset$ and $\beta$ be two disjoint subsets of $\ens{n}$. Then,
\begin{equation}\label{proplem}
\mathbb{E}\left(S(\Delta_\alpha S)^\beta\right)=\frac{1}{2^{|\alpha|}}\mathbb{E}\left(\Delta_\alpha S(\Delta_\alpha S)^\beta\right).
\end{equation}
\end{lem}
\begin{proof}
First, if $\beta = \emptyset$, then the above identity reduces to \eqref{basecasefirst}.   Next, by a straightforward induction on the cardinality $k:=|\alpha|$, note that $\Delta_\alpha S=\sum_{\alpha'\subset \alpha}(-1)^{|\alpha'|}S^{\alpha'}$. Then, for any $\alpha'\subset \alpha$,
\begin{equation*}
(-1)^{|\alpha'|}S^{\alpha'}(\Delta_\alpha S)^\beta  =\sum_{\alpha''\subset \alpha}(-1)^{|\alpha'|+|\alpha''|}S^{\alpha'}S^{\alpha''\cup \beta},
\end{equation*}
and so using \eqref{difsym} ($\alpha$ and $\beta$ are disjoint and $\alpha'\subset \alpha$ so  $\alpha'\Delta(\alpha\cup\beta)=(\alpha'\Delta \alpha) \cup \beta$),
\begin{equation*}
\esp \left((-1)^{|\alpha'|}S^{\alpha'}(\Delta_\alpha S)^\beta\right) =\esp \left(\sum_{\alpha''\subset \alpha}(-1)^{|\alpha'|+|\alpha''|}S S^{(\alpha'\Delta \alpha'')\cup \beta}\right).
\end{equation*}
Since $\alpha''\mapsto \alpha'\Delta \alpha''$ is just a permutation of the subsets of $\alpha$ and $(-1)^{ \alpha'\Delta \alpha''}=(-1)^{|\alpha'|+|\alpha''|}$,

\begin{align*}
\esp \left((-1)^{|\alpha'|}S^{\alpha'}(\Delta_\alpha S)^\beta\right) &=\esp \left(\sum_{\alpha''\subset \alpha}(-1)^{|\alpha''|}S S^{\alpha''\cup \beta}\right)=\mathbb{E}\left(S(\Delta_\alpha S)^\beta\right),
\end{align*}
and so
\begin{align*}
\frac{1}{2^{|\alpha|}}\mathbb{E}\left(\Delta_\alpha S(\Delta_\alpha S)^\beta\right) &= \frac{1}{2^{|\alpha|}}\sum_{\alpha'\subset \alpha}\esp \left((-1)^{|\alpha'|}S^{\alpha'}(\Delta_\alpha S)^\beta\right)=\mathbb{E}\left(S(\Delta_\alpha S)^\beta\right).
\end{align*}
\end{proof}

Let $T$ be the forward shift operator, i.e., for $k\in\ens{n-1}$, let $TB_k:=B_{k+1}$  and let $D$ be the backward discrete derivative: $D:=Id-T$ (so for $k\in\ens{n-1}$, $DB_k=B_k-B_{k+1}$), and denote by $D^\ell$ ($\ell\geq 1$) its $\ell$-th  iteration, with $D^0 = Id$.   It is known (see \cite{lugosi}) that the finite sequence $(B_k)_{1\leq k\leq n}$ is non-increasing. More can be said.

\begin{thm}\label{complmon}
For all $\ell\in \{0,\dots,n-1\}$ and $k\in\ens{n-\ell}$, 
\begin{equation}\label{eqdb}
D^\ell B_k=\esp \frac{1}{2^{\ell+1}n!}\sum_{i\in\mathfrak{S}_n}(\Delta_{i_1,\dots,i_{\ell+1}}S)(\Delta_{i_1,\dots,i_{\ell+1}}S)^{i_{\ell+2},\dots,i_{k+\ell}}.
\end{equation}
In particular, $D^\ell B_k\geq 0$, i.e., $(B_k)_{1\leq k\leq n}$ is completely monotone (recall that $D=Id-T$).
\end{thm}

\begin{proof}[Proof]
With the previous lemma, it is enough to prove that for all $\ell\in\{0,\dots,n-1\}$ and $k\in\ens{n-\ell}$, \begin{equation}\label{DBS}
D^\ell B_k=\esp \frac{1}{n!}\sum_{i\in\mathfrak{S}_n}S(\Delta_{i_1,\dots,i_{\ell+1}}S)^{i_{\ell+2},\dots,i_{k+\ell}}.
\end{equation}
This is done by induction on $\ell$. When $\ell=0$, \eqref{DBS} is just the very definition of $B_k$. Assume next that \eqref{DBS} holds for $\ell\in\{0,\dots,n-2\}$. Let $k\in\ens{n-(\ell+1)}$. Then,
\begin{align*}
D^{\ell+1} B_k&=\esp \frac{1}{n!}\sum_{i\in\mathfrak{S}_n}S(\Delta_{i_1,\dots,i_{\ell+1}}S)^{i_{\ell+2},\dots,i_{k+\ell}}-\esp \frac{1}{n!}\sum_{i\in\mathfrak{S}_n}S(\Delta_{i_1,\dots,i_{\ell+1}}S)^{i_{\ell+2},\dots,i_{k+1+\ell}}\\
&=\esp \frac{1}{n!}\sum_{i\in\mathfrak{S}_n}S(\Delta_{i_1,\dots,i_{\ell+1}}S)^{i_{\ell+3},\dots,i_{k+1+\ell}}-\esp \frac{1}{n!}\sum_{i\in\mathfrak{S}_n}S(\Delta_{i_1,\dots,i_{\ell+1}}S)^{i_{\ell+2},\dots,i_{k+1+\ell}}\\
&=\esp \frac{1}{n!}\sum_{i\in\mathfrak{S}_n}S(\Delta_{i_1,\dots,i_{\ell+2}}S)^{i_{\ell+3},\dots,i_{k+1+\ell}},
\end{align*}
where in getting the second equality, the terms are reindexed.
\end{proof}

We wish now to study potential connections between the $B_k$'s and jackknives operators $J_k$ and $K_k$ previously studied in \cite{jackvar2020}. For $Y\in\sigma(X_1,\dots,X_n)$, i.e., Y measurable with respect to the $\sigma$-field generated by $X_1,\dots,X_n$ and $i\in\ens{n}$, let $\esp^{(i)}Y:=\esp(Y|X_1,\dots,X_{i-1},X_{i+1},\dots,X_n)$ (with the convention $\esp^{(0)}Y:= Y$) and more generally for any subset $\alpha$ of $\ens{n}$, let 
$$\esp^{\alpha}Y:=\esp(Y|(X_j)_{j\notin \alpha}),$$
with $\esp^{\emptyset}Y = Y$, and $\esp^{\ens{n}}Y = \esp Y$.   Next, 
for $i\in\ens{n}$, let  $$\Var^{(i)}Y:=\esp^{(i)}Y^2-(\esp^{(i)}Y)^2$$ while iterating, for $i\in \mathfrak{S}_n$, and $k \ge 1$, let $$\Var^{(i_1,\dots,i_k)}Y:=\esp^{(i_1)}(\Var^{(i_2,\dots,i_k)}Y)-\Var^{(i_2,\dots,i_k)}(\esp^{(i_1)}Y),$$ 
and so (using the commutation property of conditional expectations) for any $\alpha = \{i_1, \dots, i_k\}$, 
$\Var^\alpha Y$ is well defined.  

For $k\in\ens{n}$, let
\begin{equation}
J_k:=\esp\sum_{i_1\neq i_2\dots\neq i_k}\Var^{(i_1,\dots,i_k)}S,
\end{equation}
and
\begin{equation}
K_k:=\esp\sum_{i_1\neq i_2\dots\neq i_k}\Var^{(i_1,\dots,i_k)}\esp ^{\overline{(i_1,\dots,i_k)}}S,
\end{equation}
where $\overline{(i_1,\dots,i_k)}=(i_{k+1},\dots,i_n)$.
For ease of notation, set also $J'_k:=J_k/k!$ and  $K'_k:=K_k/k!$. The next lemma provides relationships between these quantities and the $B_k$'s, it allows to easily get, and in a unified fashion, many of the known expressions involving the variance, along with some new ones.

\begin{lem}
Let  $\alpha\ne\emptyset$ and $\beta$ be two disjoint subsets of $\ens{n}$. Then
\begin{equation}\label{eq:thmvaresp}
\esp  \left(\Var^{\alpha}\esp ^{\beta}S\right)=\esp \left( S(\Delta_\alpha S)^\beta\right).
\end{equation}
\end{lem}
\begin{proof} 
If $\beta = \emptyset$, then the result is clear.  Next, the proof is done by induction on the cardinality of $\alpha$.  For the base case, let $|\alpha| = 1$, i.e.,  let $\alpha = \{i\}$, for some $i\in \ens{n}$.   Then, 
\begin{align*}
\esp  \left(\Var^{(i)}\esp ^{\beta}S\right) &= 	\esp  \left(\esp^{(i)}(\esp ^{\beta}S)^2 - (\esp^{(i)}\esp ^{\beta}S)^2 \right) \\
&= \esp(\esp ^{\beta}S)^2 - \esp(\esp^{\beta, i}S)^2 \\
&= \esp(SS^{\beta}) - \esp(SS^{\beta, i}) \\
&= \esp(S(\Delta_iS)^{\beta}), 
\end{align*}
where to get the next to last equality we use $\esp(SS^{\beta}) = \esp (\esp^\beta(SS^{\beta})) =  \esp(\esp ^{\beta}S)^2$,  which follows from the independence assumption.  Now, let us assume the validity of the lemma for any set of cardinality $k$ and let $\alpha = \{i_1, \dots, i_{k+1}\}$.   Then, since 
$\alpha$ and $\beta$ are disjoint, and by independence, 
\begin{align*}
	\esp  \left(\Var^{\alpha}\esp ^{\beta}S\right) &= 	\esp  \left(\esp^{(i_1)} (\Var^{(i_2,\dots,i_{k+1})}\esp^\beta S)-\Var^{(i_2,\dots,i_{k+1})}(\esp^{(i_1)}\esp^\beta S) \right) \\
	&=  \esp  \left(\Var^{(i_2,\dots,i_{k+1})}\esp^\beta S\right) - 
	\esp  \left(\Var^{(i_2,\dots,i_{k+1})}\esp^{(i)}S\esp^\beta S\right) \\
	&=   \esp  \left(S(\Delta_{(i_2,\dots,i_{k+1})}S)^\beta \right) - \esp  \left(S(\Delta_{(i_2,\dots,i_{k+1})}S)^{\beta,i} \right) \\
	&= \esp(S(\Delta_\alpha S)^{\beta}), 
\end{align*}
where in the middle equality we have used the induction hypothesis.

\end{proof}
Recalling \eqref{DBS}, we get from \eqref{eq:thmvaresp} that for all $k\in\ens{n}$,
\begin{equation}\label{JKD}
J'_{k}={{n}\choose{k}}D^{k-1}B_{1}\quad\text{and}\quad K'_{k}={n\choose k}D^{k-1}B_{n-k+1}.
\end{equation}

It is easy to check that for any finite sequence $(a_k)_{1\leq k\leq n}$ and any positive integers $k\in\ens{n}$, 
\begin{equation}\label{eq:inversion}
a_k= \sum_{j=0}^{k-1}(-1)^j{k-1 \choose j}D^j a_1=\sum_{j=0}^{n-k}{n-k\choose j} D^j a_{n-j}.
\end{equation}
In particular, for all $k\in\ens{n}$, 
\begin{align}
B_k &= \sum_{j=0}^{k-1}(-1)^j \frac{{k-1 \choose j}}{{n \choose j+1}}J'_{j+1}\label{eq:invBJ},\\
B_k &= \sum_{j=0}^{n-k}\frac{{n-k\choose j} }{{n \choose j+1}} K'_{j+1}\label{eq:invBK}.
\end{align}

We can now connect the $J'_k$'s and $K'_k$'s to the variance.
\begin{lem}\label{vardec}
For all $k\in\ens{n}$, 
\begin{equation}\label{formJ}
\Var S-J'_1+J'_2-\dots+(-1)^k J'_k =(-1)^k \sum_{1\leq i_1<\dots <i_{k+1}\leq n}D^k B_{i_1},
\end{equation}
\begin{equation}\label{formK}
\Var S-K'_1-K'_2-\dots-K'_k =\sum_{1\leq i_1<\dots <i_{k+1}\leq n}D^k B_{i_{k+1}}.
\end{equation}
\end{lem}
\begin{proof}
Let us prove \eqref{formJ} by induction on $k\in\ens{n}$. For the base case:

$$\Var S-J'_1=B_1+\dots+B_n-n B_1=\sum_{j=2}^n (B_j-B_1)=-\sum_{j=2}^n\sum_{i=1}^{j-1} DB_i.$$

For the inductive step: assume it is true for $k\in\ens{n-1}$. Then,
\begin{align*}
\Var S-J'_1+J'_2-\dots+(-1)^k J'_k+(-1)^{k+1}J'_{k+1} & =(-1)^k\left( \left(\sum_{1\leq i_1<\dots <i_{k+1}\leq n}D^k B_{i_1}\right)-J'_{k+1}\right)\\
& =(-1)^k \sum_{1\leq i_1<\dots <i_{k+1}\leq n}\left(D^k B_{i_1}-D^k B_1\right)\\
& =(-1)^k \sum_{1\leq i_1<\dots <i_{k+1}\leq n}\sum_{1\leq i_0<i_1}-D^{k+1} B_{i_0}\\
& =(-1)^{k+1} \sum_{1\leq i_0<i_1<\dots <i_{k+1}\leq n} D^{k+1} B_{i_0}.
\end{align*}
The proof of \eqref{formK} is very similar and so it is omitted. The following proposition recovers and extends some of the results obtained in \cite{jackvar2020}.
\end{proof}
\begin{prop}\label{corsix}
\begin{equation}\label{egalJ}
\Var S=J'_1-J'_2+\dots+(-1)^{n-1} J'_n = K'_1+K'_2+\dots+K'_n,
\end{equation}
and for all $k\in\ens{n-1}$, 
\begin{equation}
K'_{k+1}\leq (-1)^k\left(\Var S-J'_1+J'_2-\dots+(-1)^k J'_k\right)\leq J'_{k+1},
\end{equation}
\begin{equation}
K'_{k+1}\leq \Var S-K'_1-K'_2-\dots-K'_k\leq J'_{k+1}.
\end{equation}
\begin{equation}\label{restK}
\Var S= J'_1-J'_2+\dots+(-1)^{k-1}J'_k+(-1)^k\sum_{j=k+1}^{n}{j-1 \choose k}K'_{j}.
\end{equation}
\begin{equation}\label{tronc}
\Var S=\frac{{k \choose 1}}{{n \choose 1}}J'_1-\frac{{k \choose 2}}{{n \choose 2}}J'_2+\dots+(-1)^{k-1}\frac{{k \choose k}}{{n \choose k}}J'_k+\frac{{n-k \choose 1}}{{n \choose 1}}K'_1+\frac{{n-k \choose 2}}{{n \choose 2}}K'_2+\dots+\frac{{n-k \choose n-k}}{{n \choose n-k}}K'_{n-k}.
\end{equation}
\end{prop}
\begin{proof}
Above, the first two equalities simply follow from the fact that the right-hand terms in Lemma \ref{vardec} are zero when $k=n$.
Then, the first two inequalities follow from Lemma \ref{vardec} and the complete monotonicity of the $B_k$'s: for $1\leq i_1\leq n-k$, $D^kB_{n-k}\leq D^k B_{i_1}\leq D^k B_1$. Let us turn to the identity \eqref{restK}. 

Applying the inversion formula \eqref{eq:inversion} to $(D^k B_i)_{1\leq i\leq n-k}$, with $i\leq n-k$, we get
\begin{align*}
\sum_{1\leq i_1<\dots <i_{k+1}\leq n}D^k B_{i_1}&=\sum_{i=1}^{n-k}{n-i \choose k}D^k B_{i}\\
&=\sum_{i=1}^{n-k}\sum_{j=0}^{n-k-i}{n-i \choose k}{n-k-i \choose j}D^{k+j} B_{n-k-j}\\
&=\sum_{i=1}^{n-k}\sum_{j=0}^{n-k-i}{n-i \choose k}{n-k-i \choose j}\frac{K'_{k+j+1}}{{n \choose k+j+1}}\\
&=\sum_{j=0}^{n-1}\sum_{i=1}^{n-k-j}{n-i \choose k+j}{k+j \choose k}\frac{K'_{k+j+1}}{{n \choose k+j+1}}\\
&=\sum_{j=0}^{n-1}{k+j \choose k}K'_{k+j+1},
\end{align*}
where the last equality stems from the hockey-stick formula and reindexing.

To finish, let us prove \eqref{tronc} which will follow from $\Var S=B_1+\dots+B_k+B_{k+1}+\dots+B_n$. Indeed, the equality \eqref{egalJ} remains valid for any sequence $(a_n)_{n\geq 1}$, namely, the same proof shows that 
\begin{equation}
a_1+\dots+a_n={n \choose 1}D^0 a_1-{n \choose 2}D^1 a_1+\dots+(-1)^{n-1}{n \choose n}D^{n-1}a_1.
\end{equation}
In particular,
\begin{align*}
B_1+\dots+B_k & ={k \choose 1}D^0 B_1-{k \choose 2}D^1 B_1+\dots+(-1)^{k-1}{k \choose k}D^{k-1}B_1\\
 & = \frac{{k \choose 1}}{{n \choose 1}}J'_1-\frac{{k \choose 2}}{{n \choose 2}}J'_2+\dots+(-1)^{k-1}\frac{{k \choose k}}{{n \choose k}}J'_k.
\end{align*}
The second part, $B_{k+1}+\dots+B_n$, is treated similarly.
\end{proof}

The equality \eqref{tronc} could be of use to find the order of $\Var S$ as $n$ tends to infinity. For example, if there is a constant $C>1$ (independent of $n$) such that $J'_2(n)\leq C J'_1(n)$, then, taking $k=\lfloor \frac{n}{2 C}\rfloor$ will lead to $$\liminf_{n\to\infty} \frac{\Var S(n)}{J'_1(n)}\geq \frac{1}{4C}.$$

We have proved that the finite sequence $(B_k)_{1\leq k \leq n}$ is completely monotone and we already knew from \cite{lugosi} that it is non-increasing, so it is natural to wonder if one could find further properties of the $B_k$'s. On the other hand, one may also wonder whether or not $(K_k)_{1\leq k\leq n}$ does satisfy any further property except, of course, from being non-negative. Both answers appear to be negative:

\begin{prop}
For any $a_1,\dots,a_n\geq 0$, there exists $S:\bbr^n\to \bbr$ a Borel function such that for all $k\in\ens{n}$, $K_k=a_k$.
\end{prop}
\begin{cor}
If $(b_k)_{1\leq k\leq n}$ is completely monotone, then there exists $S:\bbr^n\to \bbr$ a Borel function such that for all $k\in\ens{n}$, $B_k=b_k$.
\end{cor}
\begin{proof}[Proof of the Corollary]
It is easy to see that $(b_k)_{1\leq k\leq n}$ is completely monotone if and only if for all $k\in\ens{n}$, $D^{k-1}b_{n-k+1}\geq 0$. From the statement of the proposition, there exists $S:\bbr^n\to \bbr$ a Borel function such that for all $k\in\ens{n}$, $K_k=\frac{n!}{(n-k)!}D^{k-1}b_{n-k+1}$, and, recalling \eqref{JKD}, since there is no choice for the $B_k$'s knowing the $K_k$'s, $B_k=b_k$.
\end{proof}
\begin{proof}[Proof of the proposition]
This follows from using the link with the Hoeffding decomposition observed in \cite{jackvar2020}. Consider for example $A_1,\dots,A_n\geq 0$ and $S(X_1,\dots,X_n):=A_1\sum_{1\leq i_1\leq n}(X_{i_1}-\esp X_{i_1})+A_2\sum_{1\leq i_1<i_2\leq n}(X_{i_1}-\esp X_{i_1})(X_{i_2}-\esp X_{i_2})+\dots+A_n\sum_{1\leq i_1<\dots<i_n\leq n}(X_{i_1}-\esp X_{i_1})\dots(X_{i_n}-\esp X_{i_n})$. Then, from \cite{jackvar2020}, 

 \begin{align*}
K_k & =A_k^2 \,k! \sum_{1\leq i_1<\dots<i_k\leq n}\Var(X_{i_1}-\esp X_{i_1})\dots(X_{i_n}-\esp X_{i_k})\\
& = A_k^2\, k! \sum_{1\leq i_1<\dots<i_k\leq n}\Var(X_{i_1})\dots\Var(X_{i_k}),
\end{align*}
so it is possible to adjust the $A_k$'s to have the $K_k$'s as wanted.
\end{proof}
One could expect the $J_k$'s to behave like the $K_k$'s and to also be able to take any values, but this is unfortunately not the case, for example ${2}J_2/n=(n-1)(B_1-B_2)\leq n B_1=J_1$.

To conclude this section, we connect the $B_k$'s and the quantities $T_A$ introduced in \cite{chatterjee2008new}. 
For any subset $A$ of $\ens{n}$, including $A=\emptyset$, $T_A$ is defined as 
\begin{equation}
T_A=\sum_{j\notin A} \Delta_j S (\Delta_j S)^A,
\end{equation}
and then $T$ is defined as
\begin{equation}
T=\sum_{A\subsetneq \ens{n}}\frac{T_A}{2(n-|A|){n \choose |A|}}.
\end{equation}
 It is easy to check that for all $k\in\ens{n}$,
$$B_k=\sum_{A:|A|=k-1}\frac{\esp(T_A)}{2(n-|A|){n \choose |A|}},$$
hence $\esp T=\sum_{k=1}^n B_k=\Var S$ (as expected).

\begin{rem}

\begin{enumerate}[label=(\roman*)]
\item One might wonder if the above variance results can be transferred to 
the $\Phi$-entropy.  
Let $\Phi$ be a convex function of the real variable such that $\esp |\Phi(S)| <
+\infty$, and let the $\Phi$-entropy $H_\Phi$ of $S$ (e.g., see \cite{boucheron}) be  defined as:
\begin{equation*}
H_\Phi(S)=\esp\Phi(S)-\Phi(\esp S).  
\end{equation*}
Following \cite{jackvar2020}, for $i\in\ens{n}$, let
\begin{equation*}
H^{(i)}_{\Phi} (S) = \esp^{(i)} \Phi (S) - \Phi(\esp^{(i)}(S)), 
\end{equation*}
while for $i\neq j\in\ens{n}$,
\begin{equation*}
H^{(j,i)}_\Phi(S):=\esp^{(j)}H^{(i)}_\Phi(S)-H^{(i)}_\Phi(\esp^{(j)}S)=H^{(i,j)}_\Phi(S). 
\end{equation*}
Still iterating, for $i_1\neq\dots\neq i_k\in\ens{n}$,
\begin{equation*}
H^{(i_1,\dots,i_k)}_\Phi(S):=\esp^{(i_1)}H^{(i_2,\dots,i_k)}_\Phi(S)-H^{(i_2,\dots,i_k)}_\Phi(\esp^{(i_1)}S).  
\end{equation*}
Define the corresponding  $B_k$'s as,
\begin{equation*}
B_k:=\esp \frac{1}{n!}\sum_{i\in\mathfrak{S}_n}H^{(i_k)}_\Phi(\esp^{(i_1,\dots,i_{k-1})}S),
\end{equation*}
for all $k\in\ens{n}$.  Once again the sum is telescopic:
\begin{equation*}
\sum_{k=1}^n B_k=\esp \frac{1}{n!}\sum_{i\in\mathfrak{S}_n}\esp^{(i_k)} \Phi (\esp^{(i_1,\dots,i_{k-1})}S) - \Phi(\esp^{(i_1,\dots,i_{k})}S)=H_\Phi S.
\end{equation*}
By the conditional Jensen inequality, the $B_k$'s are non-negative. Just like in the variance case, it is clear by induction that for all $\ell\in\{0,\dots,n-1\}$, 
\begin{equation*}
D^\ell B_k=\esp \frac{1}{n!}\sum_{i\in\mathfrak{S}_n}H^{(i_1,\dots,i_{\ell+1})}_\Phi(\esp^{(i_{\ell+2},\dots,i_{k+\ell})}S).
\end{equation*}
Let us now look for the class of convex functions $\Phi$ such that for any $S$ and $X_1,\dots,X_n$ satisfying the basic independence and integrability assumptions, $(B_k)_{1\leq k \leq n}$ is non-increasing. In particular, for any random variable $Z$ defined on a product space $\Omega_1\times \Omega_2$ satisfying the integrability conditions, choosing $S$ and $X_1,\dots,X_n$ such that $S=Z$ ($S=f(X_1,X_2)$ for some function $f$), we have that 
\begin{align*}
D^1 B_k & =\frac{1}{n!}\esp\sum_{i\in\mathfrak{S}_n}H^{(i_1,i_2)}_\Phi(\esp^{(i_3,\dots,i_{k+1})}S)\\
&= \frac{2}{n!}\esp H^{(1,2)}_\Phi(S)\\
& =  \frac{2}{n!}\esp\left(\Phi(Z)-\Phi(\esp^{(1)}Z)-\Phi(\esp^{(2)}Z)+\Phi(\esp^{(1,2)}Z)\right), 
\end{align*}
so $\esp \left(\Phi(Z)-\Phi(\esp^{(1)}Z)-\Phi(\esp^{(2)}Z)+\Phi(\esp^{(1,2)}Z)\right)\geq 0$. Reciprocally, if for any random variable $Y$ defined on a product space $\Omega_1\times \Omega_2$ satisfying the integrability conditions, $$\esp \left(\Phi(Y)-\Phi(\esp^{(1)}Y)-\Phi(\esp^{(2)}Y)+\Phi(\esp^{(1,2)}Y)\right)\geq 0,$$ then clearly $D^1 B_k\geq 0$ for all $k\in\ens{n-1}$. Theorem 1 in \cite{wolff} tells us that this happens if and only if $\Phi$ is affine or is twice differentiable with $\Phi''>0$ and $1/\Phi''$  concave.

\item One may further wonder what conditions on $\Phi$ would guarantee $(B_k)_{1\leq k \leq n}$ to be completely monotone, or, at least, to have $D^2 B_k\geq 0$ for all $k\in\ens{n-2}$. Unfortunately, the variance is basically the only case for which this holds true. Indeed, if the condition $D^2 B_k\geq 0$ is satisfied for all $S$, then, as before, choosing $S=f(X_1, X_2, X_3)$, we get 
\begin{align*}
D^1 B_k & =\frac{1}{n!}\esp \sum_{i\in\mathfrak{S}_n}H^{(i_1,i_2,i_3)}_\Phi(\esp^{(i_3,\dots,i_{k+2})}S)\\
&= \frac{6}{n!}\esp H^{(1,2,3)}_\Phi(S)\\
& = \frac{6}{n!}\esp \sum_{\alpha\subset\{1,2,3\}}(-1)^{|\alpha|}\Phi(\esp^{\alpha}S).  
\end{align*}
Therefore, for any random variable $Y$ defined on a product space $\Omega_1\times\Omega_2\times\Omega_3$ satisfying the integrability conditions, $\sum_{\alpha\subset\{1,2,3\}}(-1)^{|\alpha|}\Phi(\esp^{\alpha}Y)\geq 0$. Reciprocally, this guarantees the non-negativity of  $D^2 B_k$, for any $k\in\ens{n-2}$ and any $S$. According to \cite[Theorem 2]{wolff}, this happens if and only if there exist $a,b,c\in\bbr$ with $a\geq 0$ and $\Phi:x\mapsto ax^2+bx+c$. So for any function $\Phi$ that is not of this form, the $K_k$'s and the $J_k$'s (defined as the variations of $B_k$'s) are not always non-negative: for some functions $S$ they are negative.

\item It is tempting to use the representation of completely monotone functions for the $B_k's$. Unfortunately, a completely monotone finite sequence may not be the restriction of a completely monotone function.

\item From the decomposition of the variance one easily gets the decomposition of the covariance of square integrable $S$ and $T$ : $\bbr^n\to \bbr$, as a polarization identity gives: 

\begin{equation*}
\Cov(S,T)=\sum_{k=1}^n B_k(S,T),
\end{equation*}
where
\begin{align*}
B_k(S,T)=\esp \frac{1}{n!}\sum_{i\in\mathfrak{S}_n}S(T^{i_1,\dots ,i_{k-1}}-T^{i_1,\dots ,i_{k}})&=\esp \frac{1}{2n!}\sum_{i\in\mathfrak{S}_n}(S-S^{i_k})(T^{i_1,\dots ,i_{k-1}}-T^{i_1,\dots ,i_{k}}) \\
&=\esp \frac{1}{2n!}\sum_{i\in\mathfrak{S}_n}(T-T^{i_k})(S^{i_1,\dots ,i_{k-1}}-S^{i_1,\dots ,i_{k}}) .
\end{align*}

To get a symmetrical formula of the variance, let for any $i_1,\dots i_k\in \ens{n}$, $S^{\widetilde{i_1},\dots,\widetilde{i_k}}$ be as $S^{i_1,\dots,i_k}$, but with $X''_k$ in place of $X'_k$, where the $X''_k$'s have same distribution as the $X_k$'s and are independant of all previous random variables. Then,
\begin{align*}
B_k(S,T)&=\esp \frac{1}{2n!}\sum_{i\in\mathfrak{S}_n}\Delta_{i_k}(T^{\widetilde{i_1},\dots,\widetilde{i_k-1}})\Delta_{i_k}(S^{i_1,\dots ,i_{k-1}}) .
\end{align*}
\end{enumerate}
\end{rem}

\section{Connections with decompositions of the variance}

\subsection{Connection with a more general decomposition of the variance}

Let $U_1,\dots,U_n$ be random variables taking values in $(0,1)$ and independent of $X_1,\dots,X_n,X'_1,\dots,X'_n$. For any $\alpha\in [0,1]$, let $X^{(\alpha)}$ be the vector with coordinates $X^{(\alpha)}_i:=\mathds{1}_{\alpha\leq U_i} X_i+\mathds{1}_{\alpha> U_i} X'_i$, $1\leq i \leq n$. Then,
\begin{equation}
\Var S=\esp \left(S(X^{(0)})\left(S(X^{(0)})-S(X^{(1)})\right)\right),
\end{equation}
and it is tempting to rewrite this last term as an integral. Let us assume that each $U_i$ has a density $\nu_i$. For any $0\leq \alpha<\alpha'\leq 1$, denote by $A_{\alpha,\alpha'}$ the random set of indices $i\in\ens{n}$ such that $\alpha\leq U_i<\alpha'$.  Then, conditioning on the cardinality of $A_{\alpha, \alpha'}$, applying the Cauchy-Schwarz inequality and an inductive argument lead to:  
\begin{align*}
\left|\esp \left(S(X^{(0)})S(X^{(\alpha')})\right)-\esp\left(S(X^{(0)})S(X^{(\alpha)})\right)\right| & \leq 2\esp (S^2) \pr\left(|A_{\alpha,\alpha'}|>0\right)\\
&\leq 2\esp (S^2) \esp |A_{\alpha,\alpha'}|\\
&=  2\esp (S^2) \sum_{i=1}^n \int_{\alpha}^{\alpha'} \mathrm{d}\nu_i.  
\end{align*}
Therefore, $\alpha \mapsto \esp \left(S(X^{(0)})S(X^{(\alpha)})\right)$ is absolutely continuous, its derivative is well defined almost everywhere, integrable, and 
\begin{align}\label{varint}
\Var S =\esp \left(S(X^{(0)})S(X^{(0)})\right)-\esp \left(S(X^{(0)})S(X^{(1)})\right)
 =-\int_0^1 \frac{\mathrm{d}}{\mathrm{d}\alpha}\esp\left(S(X^{(0)})S(X^{(\alpha)})\right)\mathrm{d}\alpha.
\end{align}

In order to compute the derivative in \eqref{varint}, fix $\alpha\in (0,1)$ and $\varepsilon\in (0,1-\alpha)$. Conditioning on $A_{\alpha,\alpha+\varepsilon}$ and letting $$\Delta_{\alpha,\varepsilon}:=\frac{\esp \left(S(X^{(0)})S(X^{(\alpha+\varepsilon)})\right)-\esp\left(S(X^{(0)})S(X^{(\alpha)})\right)}{\varepsilon},$$ we get
\begin{align*}
\resizebox{1 \textwidth}{!} {$\Delta_{\alpha,\varepsilon}=\sum_{1\leq i_1<\dots<i_k\leq n, k\leq n} \frac{\esp \left(S(X^{(0)})\left(S(X^{(\alpha+\varepsilon)})-S(X^{(\alpha)})\right)|A_{\alpha,\alpha+\varepsilon}=\{i_1,\dots,i_k\}\right)}{\varepsilon}\pr\left(A_{\alpha,\alpha+\varepsilon}=\{i_1,\dots,i_k\}\right),$}
\end{align*}
so for almost every $\alpha$, 
\begin{equation*}
\Delta_{\alpha,\varepsilon}\xrightarrow[\varepsilon\to 0]{}\sum_{i=1}^n \esp\left(S(X^{(0)})(S(X^{(\alpha),\hat{i}})-S(X^{(\alpha),i}))\right)\nu_i(\alpha),
\end{equation*}
where $X^{(\alpha),i}$ is defined like $X^{(\alpha)}$ but with $X_i$ for its $i$-th coordinate, and $X^{(\alpha),\hat{i}}$ is defined like $X^{(\alpha)}$ but with $X'_i$ for its $i$-th coordinate. So we get finally:

\begin{equation}\label{varintgen}
\Var S = \sum_{i=1}^n \int_0^1 \esp\left(S(X^{(0)})(S(X^{(\alpha),i})-S(X^{(\alpha),\hat{i}}))\right)\mathrm{d}\nu_i(\alpha).
\end{equation}

Let us further define, for $i\in\ens{n}$ and any $x_1,\dots,x_n\in\mathbb{R}^n$, $d_i S$  via, 
\begin{equation}\label{def:di}
d_i S(x_1,\dots,x_n):=S(x_1,\dots,x_n) - \esp S(x_1,\dots,x_{i-1},X_i,x_{i+1},\dots,x_n).
\end{equation}
Note that if $Z_i$ is independent of all the other random variables and has same distribution as $X_i$, we have
\begin{equation}
d_i S(X)= \esp_{Z_i} \left(S(X)- S(X_1,\dots,X_{i-1},Z_i,X_{i+1},\dots,X_n)\right).
\end{equation}

Therefore we notice, conditioning on $U_i$, that 
\begin{equation}
\esp\left(d_i S(X^{(0)})d_i S(X^{(\alpha)})\right)=\mathds{P}(\alpha\leq U_i)\esp\left(S(X^{(0)})(S(X^{(\alpha),i})-S(X^{(\alpha),\hat{i}}))\right).
\end{equation}
We can rewrite the variance as
\begin{align}\label{vardgen}
\Var S &=\sum_{i=1}^n \int_0^1 \esp\left(d_i S(X^{(0)})d_i S(X^{(\alpha)})\right)\frac{1}{\int_\alpha^1 \mathrm{d}\nu_i(\alpha)}\mathrm{d}\nu_i(\alpha).
\end{align}
Note that in the special case where $U_i$ are uniformly distributed on $[0,1]$,
\begin{align}\label{varduni}
\Var S &=\sum_{i=1}^n \int_0^1 \esp\left(d_i S(X^{(0)})d_i S(X^{(\alpha)})\right)\frac{1}{1-\alpha}\mathrm{d}\alpha,
\end{align}
and a simple change of variables allows us to recover again \eqref{vardgen}.  
Therefore, we will focus on the uniformly distributed case.

Next, from \eqref{varintgen},
\begin{align*}
\Var S &=\sum_{i=1}^n \int_0^1 \esp\left(S(X^{(0)})(S(X^{(\alpha),i})-S(X^{(\alpha),\hat{i}}))\right)\mathrm{d}\alpha\\
&=\sum_{i=1}^n \int_0^1 \sum_{k=0}^{n-1} \esp\left(S(X^{(0)})(S(X^{(\alpha),i})-S(X^{(\alpha),\hat{i}}))\right)\mathds{1}_{|A_{0,\alpha} \setminus\{i\}|=k}\mathrm{d}\alpha\\
&=\int_0^1 n \sum_{k=0}^{n-1}\frac{1}{n}\sum_{i=1}^n \esp\left(S\left(\Delta_i S\right)^{\beta_{k,i}}\right)\pr\left(|A_{0,\alpha} \setminus\{i\}|=k\right)\mathrm{d}\alpha,
\end{align*}
where $\beta_{k,i}$ is a random set of $k$ elements chosen in $\ens{n}\setminus\{i\}$. Clearly $\pr\left(|A_{0,\alpha} \setminus\{i\}|=k\right)={n-1 \choose k}\alpha^k(1-\alpha)^{n-1-k}$, and from the representations \eqref{propB} and \eqref{proplem}, we get for any $k\in\{0,\dots,n-1\}$,
$$\frac{1}{n}\sum_{i=1}^n \esp\left(S\left(\Delta_i S\right)^{\beta_{k,i}}\right)=B_{k+1}.$$
Hence, 
\begin{align*}
\Var S &=\sum_{k=0}^{n-1} \int_0^1 n{n-1 \choose k}\alpha^k(1-\alpha)^{n-1-k}B_{k+1}\mathrm{d}\alpha=\sum_{k=0}^{n-1} B_{k+1}.
\end{align*}

\subsection{Connection with a semigroup approach}

The semigroup approach, as developed in \cite{ivanisvili2020rademacher} for the hypercube,  boils down to the same integration along $\alpha$ trick. We need first to rewrite our results in a more general setup: we assume the $X_i$'s to be i.i.d.\ discrete variables, taking a finite number of values and this time, $S$ takes values in a Banach space $(E,\|\cdot\|_E)$. We also consider a continuous convex function $\Phi:E\rightarrow \mathbb{R}^+$, so instead of considering $\Var S = \esp \| S-\esp S \|_{E}^2= \| S-\esp S \|_{E,2}^2$, we consider $\esp \left(\Phi(S-\esp S)\right)$. The price to pay is a suboptimal constant, as seen next, and the lack of connection with the $B_k$'s, which do not seem to have any equivalent in this setup. We hope that making this connection casts a new light on the breakthrough \cite{ivanisvili2020rademacher}, but also gives prospects to generalize it: indeed, while it is not clear what would be the adequate semigroup when the $X_i$'s are not binary variables, our theorem works for all discrete distributions with finite support (and it is straightforward to generalize to all discrete distributions or even bounded continuous distributions).  It should also be pointed that, motivated by geometric applications,  the case of the biased cube was also recently studied by different methods in \cite{eskenazis2024bias}.

\begin{thm}\label{thm:newhandel}
For any $\alpha\in (0,1)$, let $\varepsilon_1(\alpha),\dots, \varepsilon_n(\alpha)$ be i.i.d.\ random variables such that $\pr(\xi_i(\alpha)=1)=1-\alpha, \pr(\xi_i(\alpha)=-1)=\alpha$, and let $\delta_i(\alpha)=({\xi_i(\alpha)-\esp \xi_i(\alpha)})/{\sqrt{\Var \xi_i(\alpha)}}$. Then,
\begin{align}\label{finalnewhandel}
\esp (\Phi(S-\esp S)) & \leq  \int_0^1 \esp \Phi\left(\pi\sum_{i=1}^n \delta_i(\alpha) d_iS(X)\right)\frac{\mathrm{d}\alpha}{\pi\sqrt{\alpha(1-\alpha)}}.
\end{align}

\end{thm}

\begin{proof}

Firstly, without loss of generality, we may assume $\esp S=0$ (one may check all 
the following results are true when one adds a constant to $S$).  Following \cite{ivanisvili2020rademacher}, denoting by $\Phi^*$ the convex conjugate 
of $\Phi$, we note that for any $x\in E$,
\begin{equation}
\Phi(x)=\sup_{y\in E^*} \left(\langle y,x\rangle - \Phi^*(y)\right),
\end{equation}
and therefore, since the $X_i$'s only take a finite number of values,
\begin{equation}\label{convexphi}
\esp (\Phi(S-\esp S))=\sup_{T \text{ is }\sigma(X_1,\dots,X_n)-\text{measurable, taking values in } E^*} \esp (\langle T,S \rangle - \Phi^*(T)).
\end{equation}

\noindent
We now bound the term $\esp (\langle T,S \rangle - \Phi^*(T))$.  As in \eqref{varint},
\begin{align}
\esp (\langle T,S \rangle - \Phi^*(T)) 
 =-\int_0^1 \frac{\mathrm{d}}{\mathrm{d}\alpha}\left(\esp (\langle T,S(X^{(\alpha)})\rangle\right)\mathrm{d}\alpha - \esp \Phi^*(T)),
\end{align}
and just like in obtaining \eqref{varintgen}, we get
\begin{equation}   \label{varintgen2}
\esp (\langle T,S \rangle - \Phi^*(T)) =  \int_0^1 \esp (\langle T,\sum_{i=1}^n S(X^{(\alpha),i})-S(X^{(\alpha),\hat{i}})) \rangle \mathrm{d}\alpha- \esp \Phi^*(T).
\end{equation}

\noindent
Note that 
\begin{equation}\label{varintan}
    S(X^{(\alpha),i})-S(X^{(\alpha),\hat{i}})=d_iS(X^{(\alpha),i})-d_iS(X^{(\alpha),\hat{i}}),
\end{equation}
and by independence, 
\begin{align}
    \esp (\langle T, d_i S(X^{(\alpha),\hat{i}}) \rangle ) & = 0,
\end{align}
so
\begin{equation}   \label{varintgen3}
\esp (\langle T,S \rangle - \Phi^*(T)) = \sum_{i=1}^n \int_0^1 \esp (\langle T,d_i S(X^{(\alpha),i}) \rangle \mathrm{d}\alpha- \esp \Phi^*(T).
\end{equation}

\noindent
Now, let 
\begin{equation}
\delta_i(\alpha):=\frac{\mathds{1}_{U_i\geq \alpha} -(1-\alpha)}{\sqrt{\alpha(1-\alpha)}}=\frac{2(\mathds{1}_{U_i\geq \alpha}-1/2) -(1-2\alpha)}{2\sqrt{\alpha(1-\alpha)}},
\end{equation}
where the last equality is here to show that this is just a renormalized random variable taking values in $\{-1,1\}$, much like the $\xi_i(t)$'s, random variables with $\pr (\xi_i(t)=1)=(1+\mathrm{e}^{-t})/2$ and $\pr (\xi_i(t)=-1)=(1-\mathrm{e}^{-t})/2$ introduced in \cite{ivanisvili2020rademacher}. We have:
\begin{align}
\esp (\langle T,\delta_i(\alpha) d_iS(X^{(\alpha)})\rangle )&=\esp \left(\langle T,\frac{\mathds{1}_{U_i\geq \alpha}d_i S(X^{(\alpha),i})-(1-\alpha)(\mathds{1}_{U_i\geq \alpha}d_i S(X^{(\alpha),i})+\mathds{1}_{U_i< \alpha}d_i S(X^{(\alpha),\hat{i}}))}{\sqrt{\alpha(1-\alpha)}}\rangle \right)\\
&=\esp \left(\langle T,\frac{\alpha \mathds{1}_{U_i\geq \alpha}d_i S(X^{(\alpha), i})}{\sqrt{\alpha(1-\alpha)}}\rangle \right)\\
&=\esp \left(\langle T,\sqrt{\alpha(1-\alpha)} d_i S(X^{(\alpha), i})\rangle \right),
\end{align}
hence, with \eqref{varintgen3} we get:
\begin{align}   \label{varintgen4}
\esp (\langle T,S \rangle - \Phi^*(T)) &= \sum_{i=1}^n \int_0^1 \frac{\esp (\langle T,\delta_i(\alpha) d_iS(X^{(\alpha)})\rangle )}{\sqrt{\alpha(1-\alpha)}}\mathrm{d}\alpha- \esp \Phi^*(T)\\
&=\int_0^1 \esp (\langle T, \pi\sum_{i=1}^n \delta_i(\alpha) d_iS(X^{(\alpha)})\rangle )- \esp \Phi^*(T)\frac{\mathrm{d}\alpha}{\pi\sqrt{\alpha(1-\alpha)}}\\
&\leq \int_0^1 \esp \Phi\left(\pi\sum_{i=1}^n \delta_i(\alpha) d_iS(X^{(\alpha)})\right)\frac{\mathrm{d}\alpha}{\pi\sqrt{\alpha(1-\alpha)}}.
\end{align}

Note that $(U_1,\dots,U_n,X^{(\alpha)}_1,\dots, X^{(\alpha)}_n)$ has the same distribution as $(U_1,\dots,U_n,X_1,\dots, X_n)$, so

\begin{align}
\esp (\langle T,S \rangle - \Phi^*(T)) & \leq  \int_0^1 \esp \Phi\left(\pi\sum_{i=1}^n \delta_i(\alpha) d_iS(X)\right)\frac{\mathrm{d}\alpha}{\pi\sqrt{\alpha(1-\alpha)}}.
\end{align}
Recalling \eqref{convexphi}, the result follows.
\end{proof}

Let us see how the above allows to link our approach with the main results of \cite{ivanisvili2020rademacher} for Rademacher random variables.  Recalling the notation of  \cite{ivanisvili2020rademacher}: for $x\in \{-1,1\}^n$, let

\begin{equation}
D_iS(x):=\frac{S(x_1,\dots,x_i,\dots,x_n)-S(x_1,\dots,-x_i,\dots,x_n)}{2}.
\end{equation}

We can now state the corollary, in the Rademacher case:
\begin{cor} Let $X_i$'s be i.i.d.~Rademacher random variables,
\begin{align}\label{finalnewhandelD}
\esp (\Phi(S-\esp S)) &\leq \int_0^1 \esp \Phi\left(\pi\sum_{i=1}^n \delta_i(\alpha) D_iS(X)\right)\frac{\mathrm{d}\alpha}{\pi\sqrt{\alpha(1-\alpha)}},
\end{align}
i.e.,
\begin{align}
\esp (\Phi(S-\esp S)) &\leq \int_0^{+\infty} \esp \Phi\left(\pi\sum_{i=1}^n \delta_i(\mathrm{e}^{-2t}) D_iS(X)\right)\frac{2\mathrm{d}t}{\pi\sqrt{\mathrm{e}^{2t}-1}}.
\end{align}
\end{cor}

\begin{proof}

Since $\esp^{(i)}S$ does not depend on $x_i$, 
\begin{align*}
D_iS(x)&=\frac{(S-\esp^{(i)}S)(x_1,\dots,x_i,\dots,x_n)-(S-\esp^{(i)}S)(x_1,\dots,-x_i,\dots,x_n)}{2}\\
&=\frac{d_i S(x_1,\dots,x_i,\dots,x_n)-d_i S(x_1,\dots,-x_i,\dots,x_n)}{2}\\
&=d_iS(x),
\end{align*}
where, above, the last equality follows $d_iS(x_1,\dots,1,\dots,x_n)=-d_iS(x_1,\dots,-1,\dots,x_n)$ since $\esp^{(i)}(d_iS(X))=0$.

\end{proof}

The above implies a slightly weaker \cite[Theorem 1.2]{ivanisvili2020rademacher}, i.e., with a different absolute constant, but the fact that Enflo type and Rademacher type coincide still follows from Theorem \ref{thm:newhandel} just as  it follows from \cite[Theorem 1.4]{ivanisvili2020rademacher} with, as indicated there, a routine symmetrization argument.

To make the connection complete, recall the additional notations in \cite{ivanisvili2020rademacher}: the operator $\Delta$ is defined by
\begin{equation}
\Delta:=\sum_{i=1}^n D^i,
\end{equation}
and the semigroup $P_t$ is defined as 
\begin{equation}
P_t:=\mathrm{e}^{-t\Delta}.
\end{equation}

When the $X_i$'s are Rademacher random variables, the crucial observation in \cite{ivanisvili2020rademacher} is that (we denote by $\xi'$, $\delta'$ the variables $\xi$, $\delta$ introduced there, to avoid any confusion with $\delta$ previously defined):

\begin{equation}\label{crucialhandel}
-\frac{\mathrm{d}P_t S}{\mathrm{d}t}=\frac{1}{\sqrt{\mathrm{e}^{2t}-1}}\esp_{\xi'(t)}\left(\sum_{i=1}^n \delta'_i(t) D_i S(\xi'(t) X) \right),
\end{equation}
where $\xi'(t) X$ is defined as $(\xi'_1 (t)  X_1, \dots, \xi'_n (t)  X_n)$.

Something similar holds in a more general framework (when the $X_i's$ are random variables taking a finite number of values):

\begin{thm}
With the same assumptions as in Theorem~\ref{thm:newhandel},
\begin{equation}
-\frac{\mathrm{d}\esp_{X',U}S(X^{(\alpha)})}{\mathrm{d}\alpha}=\frac{1}{\sqrt{\alpha(1-\alpha)}}\esp_{X',U}\left(\sum_{i=1}^n \delta_i(\alpha) d_i S(X^{(\alpha)}) \right).
\end{equation}
\end{thm}
\begin{proof}
This is essentially the same proof as the proof of Theorem~\ref{thm:newhandel}.
\end{proof}

We conclude this section with a remark on the Talagrand $L_1-L_2$ inequality in Banach spaces of Rademacher type 2.

As noted in \cite{cordero2023talagrand}, it is natural, to try to understand for which Banach spaces $(E,\|\cdot\|_E)$ there exists $C=C(E)>0$ such that for any function $S$ of i.i.d.~Rademacher 
random variables $X_1,\dots,X_n$ taking values in $E$,
\begin{equation}\label{eq:talgen}
\| S-\esp S \|_{E,2}^2 \leq C \sigma(S) \sum_{i=1}^n \frac{\|D_i S\|_{E,2}^2}{1+\log\left(\frac{\|D_i S\|_{E,2}}{\|D_i S\|_{E,1}}\right)},
\end{equation}
where $\|\cdot\|_{E,k}=\left(\esp \|\cdot\|_{E}^k\right)^{1/k}$, 
which is a generalization of Talagrand's $L_1-L_2$ inequality (see Theorem \ref{thm:talag}) to Banach spaces.

Clearly, if a Banach space satisfies \eqref{eq:talgen}, it must be of Rademacher type 2. It is still unknown whether or not the converse is true.  To date, the best result is {\cite[Theorem 1]{cordero2023talagrand}}:

\begin{thm}
Let $(E,\|\cdot\|_E)$ be a Banach space of Rademacher type 2. Then there there exists $C=C(E)>0$ such that for any function $S$ of the i.i.d.~Rademacher random variables $X_1,\dots,X_n$ taking values in $E$,
\begin{equation*}
\| S-\esp S \|_{E,2}^2 \leq C \sigma(S) \sum_{i=1}^n \frac{\|D_i S\|_{E,2}^2}{1+\log\left(\frac{\|D_i S\|_{E,2}}{\|D_i S\|_{E,1}}\right)},
\end{equation*}
where $\sigma(S)=\max_{i\in\ens{n}} \log\left(1+\log\left(\frac{\|D_i S\|_{E,2}}{\|D_i S\|_{E,1}}\right)\right) $.
\end{thm}

It is still unclear whether or not the logarithmic term $\sigma(S)$ is needed, but we now show how hypercontractivity can come short to removing it.

As noted in \cite{cordero2023talagrand}, one may apply \eqref{crucialhandel} to $P_{t}S$ instead of $S$ (for a fixed $t$), while the chain rule and semigroup properties give: 
\begin{equation}
-\frac{\mathrm{d}P_{2t} S}{\mathrm{d}t}=\frac{2}{\sqrt{\mathrm{e}^{2t}-1}}\esp_{\xi'(t)}\left(\sum_{i=1}^n \delta'_i(t) D_i P_t S(\xi'(t) X) \right).
\end{equation}

Hence, since $E$ is of Rademacher type 2, denoting by $K$ its constant, we get (see e.g. \cite[(57)]{cordero2023talagrand}):

\begin{equation}\label{eq:cordhyper}
\| S-\esp S \|_{E,2} \leq 4 K \int_0^{+\infty}\left(\sum_{i=1}^n \|D_i P_t S \|_{E,2}^2\right)^{1/2}\frac{\mathrm{d}t}{\sqrt{\mathrm{e}^{2t}-1}}.
\end{equation}

We now show that in some cases hypercontractivity may not be enough to get rid of the factor $\sigma(S)$.  (A different (related?) approach for Boolean functions is presented in  \cite{ivanisvilistone}.)   More precisely, let
\begin{equation}
I=\int_0^{+\infty}\left(\sum_{i=1}^n \|D_i S \|_{E,2}^2 \left(\frac{\|D_i S \|_{E,1}}{\|D_i S \|_{E,2}}\right)^{2\frac{1-\mathrm{e}^{-2t}}{1+\mathrm{e}^{-2t}}}\right)^{1/2}\frac{\mathrm{d}t}{\sqrt{\mathrm{e}^{2t}-1}},
\end{equation}
which is the upper bound on the right term of  \eqref{eq:cordhyper} one gets using hypercontractivity.

We let $L_i= \log \left( \mathrm{e} \frac{\|D_i S \|_{E,2}}{\|D_i S \|_{E,1}}\right)$, $d_i=\|D_i S \|_{E,2}$ and $\theta(t)=\frac{1-\mathrm{e}^{-2t}}{1+\mathrm{e}^{-2t}}$, so
\begin{equation}
I \sim \int_0^{+\infty}\left(\sum_{i=1}^n d_i^2 \mathrm{e}^{-2L_i\theta(t)}\right)^{1/2}\frac{\mathrm{d}t}{\sqrt{\mathrm{e}^{2t}-1}}.
\end{equation}
With a change of variables,
\begin{align*}
I & \sim \int_0^{1}\left(\sum_{i=1}^n d_i^2 \mathrm{e}^{-2L_i \theta }\right)^{1/2}\frac{\mathrm{d}\theta}{\sqrt{\theta(1-\theta)}}\\
&\sim \int_0^{1/2}\left(\sum_{i=1}^n d_i^2 \mathrm{e}^{-2L_i \theta }\right)^{1/2}\frac{\mathrm{d}\theta}{\sqrt{\theta}}+\frac{\mathrm{e}^{-1}}{\sqrt{2}}\int_{1/2}^{1} \frac{\mathrm{d}\theta}{\sqrt{1-\theta}}\sqrt{\sum_{i=1}^n \frac{d_i^2}{L_i}},
\end{align*}
so bounding $\frac{I^2}{\sum_{i=1}^n d_i^2/L_i}$ (we already know it is bounded by $\sigma(S)$) is equivalent to bounding \begin{equation}
R:=\frac{\int_0^{1/2}\left(\sum_{i=1}^n d_i^2 \mathrm{e}^{-2L_i \theta }\right)^{1/2}\frac{\mathrm{d}\theta}{\sqrt{\theta}}}{\sqrt{\sum_{i=1}^n d_i^2/L_i}}.
\end{equation} Letting $\lambda_i:=\frac{d_i^2/L_i}{\sum_{i=1}^n d_i^2/L_i}$, we get

\begin{equation}
R=\sqrt{2}\int_0^{1/2}\left(\sum_{i=1}^n \lambda_i L_i \mathrm{e}^{-L_i \theta }\right)^{1/2}\frac{\mathrm{d}\theta}{\sqrt{\theta}}.
\end{equation}

Assume $L_i=2^{i-1}$, $\lambda_i=1/n$. Then for any $\theta\in (1/2^n,1)$, there exists $i_0\in\ens{n}$ such that $1/2^{i_0}\leq \theta\leq 1/2^{i_0-1}$, and 
\begin{align*}
\left(\sum_{i=1}^n \lambda_i L_i \mathrm{e}^{-L_i \theta }\right)^{1/2}  \geq \left(\lambda_{i_0} L_{i_0}\mathrm{e}^{-L_{i_0} \theta }\right)^{1/2} \geq \left(\frac{L_{i_0}\theta \mathrm{e}^{-L_{i_0} \theta }}{n\theta} \right)^{1/2} \geq \frac{\sqrt{2\mathrm{e}^{-2}}}{\sqrt{n\theta}},
\end{align*}
so
\begin{align*}
R \geq \frac{2\sqrt{\mathrm{e}^{-2}}}{\sqrt{n}}\int_{1/2^n}^1 \frac{\mathrm{d}\theta}{\theta}\geq 2\sqrt{\mathrm{e}^{-2}} \log(2)\sqrt{n} \geq \frac{2\sqrt{\mathrm{e}^{-2}}}{\log(2)} \sqrt{\max_{i\in\ens{n}} L_i}.
\end{align*}
Thus in this case, $\frac{I^2}{\sum_{i=1}^n d_i^2/L_i}$ is lower bounded by $C \sigma(S)$, for some constant $C>0$.

\section{Further applications to some generic inequalities}

\subsection{Iterated gradients and Gaussian (in)equalities}\label{sec:gausspoin}

It is well known that one can transfer the finite samples results of the previous section to functions of normal random variables, somehow reversing the analogies between iterated jackknives and iterated gradients first unveiled in \cite{houdre1997iterated}. This transfer is then followed by a study of the infinitely divisible framework and by the semigroup approach to these inequalities. 

Let $Z$ be a standard random variable and $G$ be an absolutely continuous function. As well known, the Gaussian Poincaré inequality asserts that
\begin{equation}
\Var G(Z)\leq \esp  \left(G'(Z)^2\right),
\end{equation}
while in \cite{jackvar1995}, this inequality is generalized with higher order gradients.
Lemma~\ref{vardec} and Proposition~\ref{corsix} allows us to quickly recover Gaussian results. Indeed, e.g., see \cite{boucheron} in the case $k=1$, one can infer from the discrete decomposition of the variance a decomposition for $\Var G(Z)$.

\begin{lem}\label{relJKder}
Let $G$ be a real-valued $m$-times continuously differentiable function, such that $\esp \left(G^{(k)}(Z)^2\right)<+\infty$, $k=0,\dots,m$.  Let $X_1,\dots, X_n,X'_1,\dots,X'_n$ be independent Rademacher random variables and let $S(X_1,\dots,X_n)\!:=G\left(\frac{X_1+\dots+X_n}{\sqrt{n}}\right)$. Then for all $k\in\ens{m}$, 
\begin{equation}
J_k(n) \xrightarrow[n \to +\infty]{} \esp  \left(G^{(k)}(Z)^2\right)\quad\text{and}\quad K_k(n) \xrightarrow[n \to +\infty]{} \left(\esp  \left(G^{(k)}(Z)\right)\right)^2.
\end{equation}
\end{lem}
\begin{proof}
It is enough to prove the theorem for $G$ $m+1$-times continuously differentiable with compact support.  From \eqref{JKD}, we have
\begin{equation}\nonumber
J_k=k!{n \choose k}D^{k-1}B_1,
\end{equation}
so (using \eqref{eqdb}),
\begin{equation}\nonumber
J_k=k!{n \choose k} \esp \frac{1}{2^{k}n!}\sum_{i\in\mathfrak{S}_n}\left(\Delta_{i_1,\dots,i_{k}}S\right)^2.
\end{equation}
By symmetry of the function $S(X_1,\dots,X_n)=G\left(\frac{X_1+\dots+X_n}{\sqrt{n}}\right)$, this simplifies to 
\begin{equation}\label{eq:Jk}
J_k=k!{n \choose k} \esp \frac{1}{2^{k}}\left(\Delta_{1,\dots,k}S\right)^2.
\end{equation}
For any $i\in\ens{k}$, 
\begin{equation}\nonumber
\Delta_i S=(D_i S) 2\mathds{1}_{X_i=X'_i}
\end{equation}
with 
$$D_iS(x):=\frac{S(x_1,\dots,x_i,\dots,x_n)-S(x_1,\dots,-x_i,\dots,x_n)}{2}.$$
Iterating,
\begin{equation}\nonumber
\Delta_{1,\dots,k} S=(D_{1,\dots,k} S) 2^k\mathds{1}_{X_1=X'_1,\dots,X_k=X'_k},
\end{equation}
hence
\begin{equation}\label{eq:DeltaD}
\esp \frac{1}{2^{k}}\left(\Delta_{1,\dots,k}S\right)^2=\esp \left(D_{1,\dots,k}S\right)^2.
\end{equation}
We now expand, for any $x\in\{-1,1\}^n$, $D_{1,\dots,k} S(x)$. Let us denote for $A\subset \ens{k}$, 
$$x^A:=(2\mathds{1}_{1\in A}-1,\dots,2\mathds{1}_{k\in A}-1, x_{k+1},\dots,x_n).$$
It is straightforward to prove by induction that
\begin{equation}\nonumber 
D_{1,\dots,k} S(x)=(-1)^{|i\in \ens{k}:x_i=1|}\frac{1}{2^k}\sum_{A\subset \ens{k}} (-1)^{|A|}S(x^A),
\end{equation}
which simplifies to 
\begin{equation}\nonumber
D_{1,\dots,k} S(x)=(-1)^{|i\in \ens{k}:x_i=1|}\frac{1}{2^k}\sum_{i=0}^k {k \choose i} (-1)^{i}G\left(\frac{2i-k+x_{k+1}+\dots+x_n}{\sqrt{n}}\right).
\end{equation}
By Taylor's formula, and using the fact that $\sum_{i=0}^k {k \choose i} (-1)^{i} i^\ell/\ell!=(-1)^k\mathds{1}_{\ell=k}$ for any $\ell\in \{0,\dots,k\}$, we get that 
\begin{equation}\nonumber
\left|D_{1,\dots,k} S(x)\right|=\left|\frac{1}{\sqrt{n}^k}G^{(k)}\left(\frac{-k+x_{k+1}+\dots+x_n}{\sqrt{n}}\right)\right|+\mathcal{O}\left(\frac{1}{\sqrt{n}^{k+1}}\right),
\end{equation}
with $\mathcal{O}$ uniform in $x$ (thanks to the compact support assumption). This leads to
\begin{equation}\nonumber
\esp n^k\left(D_{1,\dots,k} S\right)^2 \xrightarrow[n\to \infty]{} \esp\left(G^{(k)}(Z)\right)^2,
\end{equation}
and using \eqref{eq:Jk} and \eqref{eq:DeltaD}, we get the desired result
\begin{equation}\nonumber
\esp J_k(n) \xrightarrow[n\to \infty]{} \esp\left(G^{(k)}(Z)\right)^2.  
\end{equation}
The other limit in the theorem is obtained in a very similar fashion.
\end{proof}

We now see, using \eqref{eq:invBJ} and \eqref{eq:invBK}, that for any fixed $k\geq 1$, 
\begin{equation}\nonumber
B_k(n) \sim_{n \to +\infty}  \frac{1}{n}\esp \left(G'(Z)^2\right)\quad\text{and}\quad  B_{n-k}(n) \sim_{n \to +\infty} \frac{1}{n}\left( \esp\left(G'(Z)\right)\right)^2.
\end{equation}

More generally, for any $a\in (0,1)$,
\begin{equation}\nonumber
B_{\lfloor an\rfloor} (n) \sim_{n \to +\infty}  \frac{1}{n}\sum_{i=0}^\infty \frac{a^i(-1)^i}{i!}\esp  \left(G^{(i+1)}(Z)^2\right).
\end{equation}

Note that
\begin{equation}\nonumber
\int_{0}^1 B_{\lfloor an\rfloor} (n) \mathrm{d}a \xrightarrow[n \to +\infty]{} \sum_{i=0}^\infty \frac{(-1)^i}{(i+1)!}\esp  \left(G^{(i+1)}(Z)^2\right)=\Var G(Z),
\end{equation}
as one could expect.

\begin{prop}Under similar assumptions on $G$, for all $k\in\ens{m}$, 
\begin{equation}\nonumber
\frac{\left(\esp  \left(G^{(k+1)}(Z)\right)\right)^2}{(k+1)!}\leq (-1)^k\left(\Var G(Z)-\esp  \left(G'(Z)^2\right)+\dots+(-1)^k \frac{\esp  \left(G^{(k)}(Z)^2\right)}{k!}\right)\leq \frac{\esp  \left(G^{(k+1)}(Z)^2\right)}{(k+1)!}.
\end{equation}
\begin{equation}\nonumber
\frac{\left(\esp  \left(G^{(k+1)}(Z)\right)\right)^2}{(k+1)!}\leq \Var G(Z)-\left(\esp  \left(G'(Z)\right)\right)^2-\frac{\left(\esp  \left(G''(Z)\right)\right)^2}{2}-\dots-\frac{\left(\esp  \left(G^{(k)}(Z)\right)\right)^2}{k!}\leq \frac{\esp  \left(G^{(k+1)}(Z)^2\right)}{(k+1)!}.
\end{equation}
\end{prop}
The above indicates that the difference between the variance and each partial sum is  
squeezed between the Cauchy-Schwarz inequality.  We may also get equalities, when $G$ is infinitely differentiable, with additional conditions.  Indeed, 
\begin{cor}
Let $G$ be a real-valued infinitely-differentiable function, such that,  for all $k\geq 0$, 
$\esp (G^{(k)}(Z))^2 < +\infty $. Then,
\begin{equation}\nonumber
\Var G(Z)=\sum_{i=1}^{+\infty}(-1)^{i-1} \frac{\esp  \left(G^{(i)}(Z)^2\right)}{i!},
\end{equation}
if and only if $\lim_{k\to\infty}\esp (G^{(k)}(Z))^2/k! = 0$, and under such a condition,  
\begin{equation}\nonumber
\Var G(Z)=\sum_{i=1}^{+\infty}\frac{\left(\esp  \left(G^{(i)}(Z)\right)\right)^2}{i!}. 
\end{equation}
For any $k\geq 1$,
\begin{equation}\label{restKder}
Var G(Z)=\esp  \left(G'(Z)^2\right)-\frac{\esp  \left(G''(Z)^2\right)}{2}+\dots+(-1)^{k-1} \frac{\esp  \left(G^{(k)}(Z)^2\right)}{k!}+(-1)^k\sum_{j=k+1}^{\infty}{j-1 \choose k}\frac{\left(\esp  \left(G^{(j)}(Z)\right)\right)^2}{j!}.
\end{equation}
For any $a\in [0,1]$,
\begin{equation}\label{troncder}
\Var G(Z)=\sum_{i=1}^{+\infty}\left((-1)^{i-1} a^i\frac{\esp  \left(G^{(i)}(Z)^2\right)}{i!}+(1-a)^i \frac{\left(\esp  \left(G^{(i)}(Z)\right)\right)^2}{i!}\right).
\end{equation}
\end{cor}
\begin{proof}
This is nothing but Lemma~\ref{relJKder} together with Proposition~\ref{corsix}. To get the last equality,  apply \eqref{tronc} to $k=\lfloor an \rfloor$.
\end{proof}

The equality \eqref{restKder} is a generalization of the equality in \cite{jackvar2020}, where $k=1$. Note that \eqref{troncder} can be rewritten as 
\begin{equation}\nonumber
\Var G(Z)=\sum_{i=1}^{+\infty}\frac{\left(\esp  \left(G^{(i)}(Z)\right)\right)^2}{i!}+
\sum_{i=1}^{+\infty}\left((-1)^{i-1} \frac{\esp  \left(G^{(i)}(Z)^2\right)}{i!}+\sum_{j\geq i} (-1)^i{j \choose i}  \frac{\left(\esp  \left(G^{(j)}(Z)\right)\right)^2}{j!}\right)a^i,
\end{equation}
which gives us the additional equality: for all $i\geq 1$,
\begin{equation}\nonumber
\frac{\esp  \left(G^{(i)}(Z)^2\right)}{i!}=\sum_{j\geq i}{j \choose i}  \frac{\left(\esp  \left(G^{(j)}(Z)\right)\right)^2}{j!}.
\end{equation}
This gives an alternative way to find \eqref{relJKder} again:
\begin{align*}
\sum_{i=k+1}^{+\infty}(-1)^{i-1}\frac{\esp  \left(G^{(i)}(Z)^2\right)}{i!} &= \sum_{j\geq i\geq k+1} (-1)^{i-1} {j \choose i}  \frac{\left(\esp  \left(G^{(j)}(Z)\right)\right)^2}{j!}\\
&= \sum_{j=k+1}^{+\infty}\left(\sum_{i=k+1}^j (-1)^{i-1} {j \choose i}  \right)\frac{\left(\esp  \left(G^{(j)}(Z)\right)\right)^2}{j!}\\
&= \sum_{j=k+1}^{+\infty}{j-1 \choose k} \frac{\left(\esp  \left(G^{(j)}(Z)\right)\right)^2}{j!}.\\
\end{align*}

Multivariable versions of the above results remain true, and in fact, so do infinite-dimensional ones on Wiener space or Poisson space or even Fock space.  In each case, what is needed is a proper definition of the gradient, e.g., see  \cite{houdre1995covariance} for some infinite dimensional setting (Wiener and Poisson spaces).   In the multivariate setting here is a small sample of results which can be easily obtained via the techniques developed to this point:  Let $m\geq 1$, let $G:\bbr^m\rightarrow \mathbb{R}$ be a smooth function (for the sake of simplicity, just assume differentiability up to the correct order, as above), and let $Z_1,\dots,Z_m$ be i.i.d.~standard normal random variables.  Now, for $k\geq 1$, let
\begin{equation}\nonumber
\theta_k=\sum_{1\leq i_1,\dots,i_k\leq m}\left(\esp\left(\frac{\partial^k G}{\partial x_{i_1}\dots x_{i_k}}(Z_1,\dots,Z_m)\right)\right)^2,
\end{equation}
and let 
\begin{equation}\nonumber
\eta_k=\sum_{1\leq i_1,\dots,i_k\leq m}\left(\esp\left(\frac{\partial^k G}{\partial x_{i_1}\dots x_{i_k}}(Z_1,\dots,Z_m)\right)^2\right).
\end{equation}
Let further $(X_{i,j})_{i \in\ens{m}, j\in\ens{n}}$ be independent Rademacher random variables and let $$S(X_{1,1},\dots,X_{m,n}):=G\left(\frac{X_{1,1}+\dots+X_{1,n}}{\sqrt{n}},\dots,\frac{X_{m,1}+\dots+X_{m,n}}{\sqrt{n}} \right).$$  Then for all $k\geq 1$, 
\begin{equation}\nonumber
J_k(n) \xrightarrow[n \to +\infty]{} \eta_k\quad\text{and}\quad K_k(n) \xrightarrow[n \to +\infty]{} \theta_k.
\end{equation}
Moreover, for all $k\geq 1$, 
\begin{equation}\nonumber
\frac{\theta_{k+1}}{(k+1)!}\leq (-1)^k\left(\Var G(Z_1,\dots,Z_m)-\eta_1+\frac{\eta_2}{2}-\dots+(-1)^k\frac{\eta_k}{k!}\right)\leq \frac{\eta_{k+1}}{(k+1)!}.
\end{equation}
\begin{equation}\nonumber
\frac{\theta_{k+1}}{(k+1)!}\leq \Var G(Z_1,\dots,Z_m) - \theta_1-\frac{\theta_2}{2}-\dots-\frac{\theta_k}{k!}\leq \frac{\eta_{k+1}}{(k+1)!}.
\end{equation}

\begin{rem}  It is well known that if  $Z_1, Z_2, \dots, Z_m$ are iid standard normal random variables, and if 
	$\|Z\|_2^2 := \sum_{k=1}^m Z_k^2$, then $\left(Z_1/\|Z\|_2,\dots,Z_m/\|Z\|_2\right)$ is uniformly distributed on the $m-1$-dimensional unit sphere.  Therefore, the above multivariate Gaussian case allows to recover and extend various variance bounds and covariance representations on the high-dimensional sphere.
\end{rem}

\subsection{The infinitely divisible case}

Let $Y$ be an infinitely divisible real-valued random variable, and $G:\mathbb{R}\rightarrow\mathbb{R}$ be a smooth function such that its derivatives of all order are well defined and $\esp \left(G^{(k)}(Y)\right)^2<+\infty$ for all $k\geq 0$. We are interested in the decomposition of the variance of $G(Y)$.

We let $(Y_t)_{t\geq 1}$ be the corresponding Lévy process (i.e.\  $Y_1$ has the same distribution as $Y$), we denote by $(b,\sigma,\nu)$ its generator (from the Lévy–Khintchine representation), and let $(Y'_t)_{t\geq 0}$,$(Y''_t)_{t\geq 0}$ be independent copies of $(Y_t)_{t\geq 0}$. For  $1\leq {\ell},m$,  let 
\begin{align}
X_{{\ell},m} & =Y_{{\ell}/m}-Y_{({\ell}-1)/m},\\
X'_{{\ell},m} & =Y'_{{\ell}/m}-Y'_{({\ell}-1)/m}, 
\end{align}
and let 
\begin{equation}
S_n=G(X_{1,n}+\dots+X_{n,n}).
\end{equation}

We now study, for any fixed $\alpha\in (0,1)$, the limit when $m$ goes to infinity of $nB_{\lfloor \alpha n \rfloor}$ (the $B_k$'s of $S_n$) where $n=2^m+1$, which allows us to recover in another way the representation of the variance from \cite{houdre1998interpolation}.

\begin{thm}
Let $\alpha\in (0,1)$. Then with the notations above,
\begin{equation}
2^m B_{\lfloor \alpha (2^m) \rfloor+1} \xrightarrow[m\to\infty]{} \esp\left(\sigma G'(Y_\alpha + Y''_{1-\alpha})G'(Y'_\alpha + Y''_{1-\alpha})+\int_{\mathbb{R}} \Delta_u G(Y_\alpha + Y''_{1-\alpha})\Delta_u G(Y'_\alpha + Y''_{1-\alpha})\mathrm{d}\nu\right),
\end{equation}
where $\Delta_u G (x)=G(x+u)-G(x)$.
\end{thm}

\begin{proof}
We first prove this fact for $\alpha$ a dyadic rational number, $\alpha= a/2^b\in (0,1)$. Let $m\geq b$ and $n=2^m$, with $m\geq b$. The proof is more convenient to write for a slightly different function: instead of computing the $B_k$'s of $S_n$ ($n=2^m$), we compute the $B_k$'s of
\begin{equation}
T_n:=G(X_{1,n}+\dots+X_{n,n}+X_{n+1,n}).
\end{equation}
Since the difference between the former and latest has order $\mathcal{O}(1/n)$, this is enough to get the desired result.  We have
\begin{align*}
B_{\lfloor \alpha n \rfloor+1} &= \esp\resizebox{0.9 \textwidth}{!} {$\left(G\left(\sum_{i=1}^{n+1} X_{i,n}\right)\left(G\left(X_1+\sum_{i=2}^{2^{m-b}a+1} X'_{i,n}+\sum_{i=2^{m-b}a+2}^{n+1} X_{i,n}\right)-G\left(\sum_{i=1}^{2^{m-b}a+1} X'_{i,n} + \sum_{i=2^{m-b}a+2}^{n+1} X_{i,n}\right)  \right)\right)$}\\
&= \esp\left(G(Z_{1/n}+Y'_\alpha+Y''_{1-\alpha})\left(G(Z_{1/n}+Y_\alpha+Y''_{1-\alpha})-G(Z'_{1/n}+Y'_\alpha+Y''_{1-\alpha})\right)\right), 
\end{align*}
where $Z$ and $Z'$ are two independent copies of $Y$, since $(Z_{1/n},Z'_{1/n}, Y_\alpha, Y'_\alpha,Y''_{1-\alpha})$ has the same distribution as $(X_1,X'_1,\sum_{i=2}^{2^{m-b}a+1} X_{i,n}, \sum_{i=2}^{2^{m-b}a+1} X'_{i,n},\sum_{i=2^{m-b}a+2}^{n+1} X_{i,n})$.

\noindent
Let $G_{\alpha,1}(\cdot)=\esp(G(\cdot + Y_\alpha+Y''_{\alpha})|Y_\alpha, Y''_{1-\alpha})$ and $G_{\alpha,2}(\cdot)=\esp(G(\cdot + Y_\alpha+Y''_{\alpha})|Y'_\alpha, Y''_{1-\alpha})$, we then have
\begin{align*}
B_{\lfloor \alpha n \rfloor+1} &= \esp\left(G_{\alpha,1}(Z_{1/n})G_{\alpha,2}(Z_{1/n})-G_{\alpha,1}(0)G_{\alpha,2}(0)-(G_{\alpha,1}(Z_{1/n})G_{\alpha,2}(Z'_{1/n})-G_{\alpha,1}(0)G_{\alpha,2}(0))\right),
\end{align*}

\noindent
so if $A_1$ be the infinitesimal generator of $(Y_t,Y_t)_{t\geq 0}$ and $A_0$ is the infinitesimal generator of $(Y_t,Y'_t)_{t\geq 0}$, then 

\begin{flalign}
nB_{\lfloor \alpha n \rfloor+1}  \xrightarrow[n\to\infty]{} &  \,\esp\left((A_1-A_0)G_{\alpha,1} \otimes G_{\alpha,2}(0,0) \right)&&\\\label{limbk}
 & = \esp\left(\sigma G'(Y_\alpha + Y''_{1-\alpha})G'(Y'_\alpha + Y''_{1-\alpha})+\int_{\mathbb{R}} \Delta_u G(Y_\alpha + Y''_{1-\alpha})\Delta_u G(Y'_\alpha + Y''_{1-\alpha})\mathrm{d}\nu\right),&&
\end{flalign}
since $(A_1-A_0)(f\otimes g)(0,0)= \sigma f'(0) g'(0) + \int_{\mathbb{R}} \Delta_u f(0) \Delta_u g (0)\mathrm{d}\nu$.  (This computation is in \cite[Proposition 2]{houdre1998interpolation}.)

Since the finite sequence $B_k$ is non-decreasing, a routine density argument shows that for any $\alpha\in (0,1)$, $nB_{\lfloor \alpha n \rfloor}$ has limit \eqref{limbk}, which is the desired result.
\end{proof}

\begin{cor}
\begin{equation}
\Var G(Y)= \int_{0}^1 \esp\left(\sigma G'(Y_\alpha + Y''_{1-\alpha})G'(Y'_\alpha + Y''_{1-\alpha})+\int_{\mathbb{R}} \Delta_u G(Y_\alpha + Y''_{1-\alpha})\Delta_u G(Y'_\alpha + Y''_{1-\alpha})\mathrm{d}\nu\right)\mathrm{d}\alpha.
\end{equation}
\end{cor}

This above representation of the variance stems from the decomposition $\Var G(Y)=\sum_{k=1}^n B_k$ and it can also be found, with a different approach, in 
\cite{houdre1998interpolation}.  Although we have only been concerned with representations of the variance, similar representations continue to hold for covariances in the spirit of the work just cited.

For example, we note that in the Poisson case, the limit \eqref{limbk} is simply 
\begin{equation}
\mathds{E} \left(DG(Y_\alpha+Y''_{1-\alpha})DG(Y'_\alpha+Y''_{1-\alpha})\right),
\end{equation}
where $DG(x)=G(x+1)-G(x)$, $Y_\alpha,Y'_{\alpha},Y''_{1-\alpha}$ are Poisson distributed independent random variables (with respective 
parameter $\alpha$, $\alpha$, and $1-\alpha$). In the Gaussian case, it is
\begin{equation}
\mathds{E}\left(G'(Z_{1,\alpha})G'(Z_{2,\alpha})\right),
\end{equation}
where $Z_{1,\alpha}$, $Z_{2,\alpha}$ are Gaussian random variables centered with variance one and covariance $1-\alpha$.

\subsection{A weaker Talagrand $L_1-L_2$ inequality}

Let us focus on the special case where the $X_i$'s are Bernoulli with  parameter $1/2$. For $i\in\ens{n}$, let \[\tau_i S(X_1,\dots,X_n):=S(X_1,\dots,X_{i-1},0,X_{i+1},\dots,X_n)-S(X_1,\dots,X_{i-1},1,X_{i+1},\dots,X_n)\] (so this does not depend on $X_i$).  Then, Talagrand's $L_1-L_2$ inequality ({\cite[Theorem 1.5]{talagrand1994russo}}) can be stated as follows:
\begin{thm}\label{thm:talag}
There exists $C>0$ such that for any function $f:\{0,1\}^n\rightarrow \bbr$, the following inequality holds
\begin{equation}
\Var S\leq C\sum_{i=1}^n \frac{\|\tau_i S\|_2^2}{1+\log\left(\frac{\|\tau_i S\|_2}{\|\tau_i S\|_1}\right)},
\end{equation}
where $S=f(X_1,\dots,X_n)$.
\end{thm}
We now prove a weaker form of this inequality using the $B_k's$, in the special case where there exists $a>0$ such that for all $i\in\ens{n}$, $|\tau_i S|\in\{0,a\}$. We can further assume without loss of generality by rescaling that $a=1$. Note that this particular case includes $LC_n$ (changing a letter can only change $LC_n$ by at most one).

Firstly, conditioning on whether $X_{i_k}=X'_{i_k}$ or $X_{i_k}\neq X'_{i_k}$, we can rewrite \eqref{propB} as
\begin{equation}\label{condB}
B_k=\esp \frac{1}{4n!}\sum_{i\in\mathfrak{S}_n}(\tau_{i_k}S)(\tau_{i_k}S)^{i_1,\dots,i_{k-1}},
\end{equation}
so
\begin{align*}
\Var S & = \sum_{k=1}^n \esp \frac{1}{4n!}\sum_{i\in\mathfrak{S}_n}(\tau_{i_k}S)(\tau_{i_k}S)^{i_1,\dots,i_{k-1}}\\
&=\frac{1}{4n!} \sum_{i\in\mathfrak{S}_n} \sum_{k=1}^n \esp (\tau_{i_1}S)(\tau_{i_1}S)^{i_2,\dots,i_k}.
\end{align*}

Let us fix $i\in \mathfrak{S}_n$ and bound $ \sum_{k=1}^n \esp (\tau_{i_1}S)(\tau_{i_1}S)^{i_2,\dots,i_k}$. For ease of notation, by reindexing the $X_i$'s, we may assume $i=Id$, also, let us write $X:=(X_2,\dots,X_n)$. Since, by assumption, $\tau_1 S$ is boolean, there exists $m\leq 2^{n-1}$ and $x^1,\dots, x^m\in \{0,1\}^{n-1}$ pairwise distinct such that $|\tau_1 S|=\sum_{i=1}^m \mathds{1}_{X=x^i}$.
Let, for $\alpha\subset \{2,\dots, n\}$, $N(\alpha):=|\{(i,j)\in\ens{m}^2 : \forall k\in\alpha, x^i_k=x^j_k\}|$. We have
\begin{equation}
\esp (\tau_1 S)(\tau_1 S)^{2,\dots,k}\leq \esp |\tau_1 S| |\tau_1 S|^{2,\dots,k}=\frac{N(\{k+1,\dots,n\})}{2^{n+k-2}}.
\end{equation}
Let ${\ell}\in\ens{n-1}$ be such that $2^{{\ell}-1}\leq m\leq 2^{\ell}$ (we may exclude the trivial case $m=0$). Using that for any $k\in\ens{{\ell}}$, $N(\{k+1,\dots,n\})\leq m 2^{k-1}$, and the trivial bound $N(\{k+1,\dots,n\})\leq m^2$ when $k>{\ell}$, we get
\begin{align*}
\sum_{k=1}^n \esp (\tau_1 S)(\tau_1 S)^{2,\dots,k} & \leq \sum_{k=1}^{\ell} \frac{m}{2^{n-1}}+\sum_{k=1+1}^n \frac{m^2}{2^{n+k-2}}\\
& \leq {\ell} \frac{m}{2^{n-1}}+2\frac{m^2}{2^{n-1+{\ell}}}\\
& \leq ({\ell}+2) \frac{m}{2^{n-1}}=({\ell}+2)\|\tau_1 S\|_2^2.
\end{align*}

Note that $\log\left(\frac{\|\tau_1 S\|_2}{\|\tau_1 S\|_1}\right)=\log\left(\sqrt{\frac{2^{n-1}}{m}}\right)={\log(2)}(n-1-\log_2(m))/2$ so ${\ell}+2\leq n+2- \frac{2}{\log(2)}\log\left(\frac{\|\tau_1 S\|_2}{\|\tau_1 S\|_1}\right)$ hence 
\begin{equation}
\sum_{k=1}^n \esp (\tau_1 S)(\tau_1 S)^{2,\dots,k} \leq \left(n+2- \frac{2}{\log(2)}\log\left(\frac{\|\tau_1 S\|_2}{\|\tau_1 S\|_1}\right)\right)\|\tau_1 S\|_2^2.
\end{equation}

Finally, 
\begin{align*}
\Var S &=\frac{1}{4n!} \sum_{i\in\mathfrak{S}_n} \sum_{k=1}^n \esp (\tau_{i_1}S)(\tau_{i_1}S)^{i_2,\dots,i_k}\\
&\leq \frac{1}{4n!} \sum_{i\in\mathfrak{S}_n} \left(n+2- \frac{2}{\log(2)}\log\left(\frac{\|\tau_{i_1} S\|_2}{\|\tau_{i_1} S\|_{1}}\right)\right)\|\tau_{i_1} S\|_2^2\\
&\leq \frac{1}{4}\sum_{j=1}^n\left(1+\frac{2}{n}- \frac{2}{n\log(2)}\log\left(\frac{\|\tau_{j} S\|_2}{\|\tau_{j} S\|_{1}}\right)\right)\|\tau_{j} S\|_2^2.\\
\end{align*}

To see that it is weaker than Talagrand's $L_1-L_2$ inequality, consider for example $X_1,\dots,X_n$ independent Bernoulli variables of parameter $1/2$, and $S$ defined on $\{0,1\}^n$ by $S(x_1,\dots,x_n):=x_1\dots x_{n/2}$ (assuming $n$ is even). Then, for any $j\in \ens{n/2}$, $\|\tau_{j} S\|_{1}=(1/2)^{\frac{n}{2}-1}$ and $\|\tau_{j} S\|_{2}=\sqrt{\|\tau_{j} S\|_{1}}$. So on the one hand, Talagrand's inequality gives a bound of order $(1/2)^{\frac{n}{2}-1}$, which is optimal, while on the other hand, our weaker bound gives an upper bound of order $n(1/2)^{\frac{n}{2}-1}$. 

\section{On the variance of the length of  longest common subsequences}

To finish these notes, we present some applications of the above inequalities to subsequences problems, in particular to lower-bounding the variance of the length of the longest common subsequences between two random words. For $(x_1,\dots,x_s)$, $(y_1,\dots,y_t)$ two sequences taking values in a finite set $\mathcal{A}$, we denote by $LCS(x_1\cdots x_s;y_1\cdots y_t)$ the largest integer $k$ such that there exists $1\leq i_1<\dots<i_k\leq s$, $1\leq j_1<\dots<j_k\leq t$ satisfying $a_{i_1}=b_{j_1},\dots ,a_{i_k}=b_{j_k}$, or $0$ if there is no such integer.  In the sequel, we take $\mathcal{A}=\ens{m}$ (for some $m$ we specify in each case), $X_1,\dots,X_n, \dots Y_1,\dots,Y_n, \dots $ i.i.d.~random variables taking values in $\mathcal{A}$ (according to a distribution we specify), and consider the length of the longest common subsequences of these two random words, written $LCS(X_1\cdots X_n;Y_1\cdots Y_n)$ or simply $LC_n$.

\subsection{A generic upper bound}

Via the Efron-Stein inequality, the upper bound $\Var S\leq nB_1$ was already obtained  \cite{steele1986efron}.   Let us study the variance of $LC_n$ with our approach.  

Let now $Z_1,\dots,Z_{2n}$ be $i.i.d.$ Bernoulli random variables of parameter $1/2$, and consider the $B_k$'s of the function $S(Z_1,\dots,Z_{2n}):=LCS(Z_1\cdots Z_{2n})=LC_n$. We know that $\Var LC_n\leq 2nB_1(2n)$. Using \eqref{condB}, $B_1(2n)\leq 1/4$ so $\Var LC_n\leq n/2$ (at result already known to \cite{steele1986efron}.  But this last bound can be slightly improved: note that by symmetry of the zeros and ones in $LC_n$ (that is, if $\bar{Z}_i:=1-Z_i, i\in\ens{2n}$, $S(Z)=S(\bar{Z})$), $\esp \tau_i S$=0 so $B_{2n}(2n)=0$. By convexity, $B_1(2n)+\dots+B_{2n}(2n)\leq 2n(B_1(2n)+B_{2n}(2n))/2$, so $\Var LC_n \leq n/4$.

In case of an alphabet $\ens{m}$, conditioning on $X_i\neq X'_i$ we get $B_1(2n)\leq \left(1-\sum_{k=1}^m p_k^2\right)/2$, and when additionally $B_{2n}(2n)=0$, then $\Var LC_n \leq \left(1-\sum_{k=1}^m p_k^2\right)n/2$, which improves, by a factor of two, on the upper bound obtained in \cite{steele1986efron} (see also  \cite{houdre2016order}). The condition $B_{2n}(2n)=0$ is realized, for instance, when $p_1=\dots=p_m=1/m$ (by symmetry).

In the remaining part of this manuscript, we focus on various lower bounds on the variance of $LC_n$.  
Since by Theorem \ref{complmon}, $(B_k)_{1\leq k\leq 2n}$ is, in particular, non-decreasing,  \begin{equation}\label{Bvar}
	\Var LC_n \geq 2n B_{2n},
\end{equation}
and this last inequality we will of use throughout this section.  As a first instance, \cite[Theorem 2.1]{lember} provides a lower bound on 
the variance of $LC_n$, in the Bernoulli case, proving that when $p$ is smaller than some universal (but unspecified and extremely small) constant, the variance is of order $n$, see also \cite{houdre2016order} for more explicit bounds. To obtain this results, the authors of \cite{lember} first show Theorem 2.2 there, and then prove that it implies a variance lower bound of order $n$.  The proof of this implication is long and we aim to show that the jackknives tools we developed greatly simplifies it. (This methodology also simplifies, in the quadratic case, the implication of \cite[Theorem 2.1]{houdre2016order} towards the linear lower bound.)  With our approach, we also generalize results for the case where one letter is omitted. We also proceed to prove, in the binary case, another slightly weaker bound: for some $p_1\in(0.096,0.5)$ (so not as small as in \cite{lember} or \cite{houdre2016order}), or $p_2\in (0.5, 0.904)$, $\limsup_{n\to +\infty}\Var LC_n/n > 0$.  Finally, we give further partial results on the order of the variance in the uniform case.

\subsection{On the order of the variance under a hypothesis on a modification of $LC_n$}

In this section we prove how Theorem 2.2 in \cite{lember} or Theorem 2.1 in \cite{houdre2016order} imply their  main theorem, namely the linear lower order of the variance. This shows how the use 
of the $B_k$'s greatly simplify some proofs, and it is of interest to infer, more generally, a lower bound on the variance from a random perturbation that has an effect on the expectation. More specifically, here, the random perturbation is to pick, in the binary case, a random $1$ from the letters (if there is at least one), and to turn it into a $0$. The original letters are denoted by $Z_1,\dots,Z_{2n}$, the new letters (with a $1$ turned into a $0$) by $\tilde{Z_1}, \dots, {\tilde Z}_{2n}$.  We refer to \cite{lember} and \cite{houdre2016order} for a more formal definition of $\tilde{Z}$. Theorem 2.2/Theorem 2.1 there implies, in particular, that for any $\delta\in (0, \alpha_1-\alpha_2)$, where $\alpha_1$, $\alpha_2$ are constants defined there such that $\alpha_1>\alpha_2$, for $n$ large enough,
\begin{equation}\nonumber
	\esp \left(LCS(Z)-LCS(\tilde{Z})\right)\geq \delta.
\end{equation}

From this, it is natural to try to prove that $B_{2n}(2n)$ is greater than some absolute constant, to infer that the variance has linear order. Let, for all $z\in\{0,1\}^{2n}$, $x\in\{0,1\}$ and $k\in\ens{2n}$, $z^{k,x}:=(z_1,\dots,z_{k-1},x,z_{k+1},\dots,z_{2n})$. Consider the modifications of $Z$, $Z^{N,1}$ and $Z^{N,0}$, with $N$ picked in $\ens{n}$ uniformly. Intuitively, this is "almost" like the previous pair $(Z,\tilde{Z})$. But it is easier to write $B_{2n}(2n)$ in terms of $\esp\left(LCS(Z^{N,1})-LCS(Z^{N,0}\right)$. Indeed, we have
\begin{align*}
B_{2n}(2n) & =\esp \frac{1}{2(2n)!}\sum_{i\in\mathfrak{S}_{2n}}(S-S^{i_{2n}})(S^{i_1,\dots ,i_{2n-1}}-S^{i_1,\dots ,i_{2n}})\\
& =\esp \frac{1}{2(2n)}\sum_{k=1}^{2n}(S-S^k)(S^{\ens{2n}\setminus\{k\}}-S^{\ens{2n}}),
\end{align*}
conditioning on $(Z_{i_{2n}}, Z'_{i_{2n}})$ (first term when its $(0,1)$, second term $(1,0)$, the other terms are null) we get
\begin{align*}
B_{2n}(2n) & =\esp \frac{1}{2(2n)}\sum_{k=1}^{2n}\left(LCS(Z^{k,0})-LCS(Z^{k,1})\right)\left(LCS(Z'^{k,0})-LCS(Z'^{k,1})\right)p(1-p)\\
&\quad\quad + \esp \frac{1}{2(2n)}\sum_{k=1}^{2n}\left(LCS(Z^{k,1})-LCS(Z^{k,0})\right)\left(LCS(Z'^{k,1})-LCS(Z'^{k,0})\right)p(1-p),
\end{align*}
and by independence, 
\begin{align*}
B_{2n}(2n) & =\frac{1}{2n}\sum_{k=1}^{2n}\left(\esp\left(LCS(Z^{k,1})-LCS(Z^{k,0})\right)\right)^2p(1-p),
\end{align*}
so by the Cauchy-Schwarz inequality,
\begin{equation}\label{CSB2n}
B_{2n}(2n) \geq \left(\esp\left(LCS(Z^{N,1})-LCS(Z^{N,0})\right)\right)^2p(1-p).
\end{equation}
We next provide a lower bound on $\esp\left(LCS(Z^{N,0})-LCS(Z^{N,1})\right)$. First note  that if $N_1$ denotes the number of 
ones, for any ${\ell}\in\ens{2n}$, $(Z^{N,1},Z^{N,0})$ conditionally on $N_1(Z^{N,1})={\ell}$ has the same distribution as $(Z,\tilde{Z})$ conditionally 
on $N_1(Z)={\ell}$. Indeed, this is the uniform distribution on all the possible pairs of $2n$ bits, the first one having $k$ ones and the second one being identical 
except exactly for a $1$ turned into a $0$. To simplify the notations, for ${\ell}\in \{0,\dots,2n\}$, let

\begin{equation}\nonumber
f({\ell}):=\esp\left(LCS(Z)-LCS(\tilde{Z})|N_1(Z)\right).
\end{equation}
We have
\begin{align*}
\esp\left(LCS(Z)-LCS(\tilde{Z})\right) & = \sum_{{\ell}=1}^{2n} f({\ell})\pr(N_1(Z)={\ell}) = \esp\left(f(N_1(Z))\right),
\end{align*}
while, since $f(0)=0$, 

\begin{align*}
\esp\left(LCS(Z^{N,1})-LCS(Z^{N,0})\right) & = \sum_{{\ell}=1}^{2n} \esp\left(LCS(Z^{N,1})-LCS(Z^{N,0})|N_1(Z)\right) \pr(N_1(Z^{N,1})={\ell})\\
&=\sum_{{\ell}=1}^{2n} f({\ell})p^{{\ell}-1}(1-p)^{n-{\ell}}{n-1\choose {\ell}-1}\\
&=\sum_{{\ell}=1}^{2n} f({\ell})\frac{{\ell}}{pn}\pr(N_1(Z)={\ell})\\
&= \esp\left(f(N_1(Z))\frac{N_1(Z)}{pn}\right),
\end{align*}
so by dominated convergence, 
\begin{equation}\nonumber
\esp\left(LCS(Z^{N,1})-LCS(Z^{N,0})\right) \xrightarrow[n \to \infty]{} \esp\left(LCS(Z)-LCS(\tilde{Z})\right). 
\end{equation}
Therefore, for any $\delta\in (0,\alpha_1-\alpha_2)$, for $n$ large enough,
\begin{equation}\nonumber
\esp\left(LCS(Z^{N,1})-LCS(Z^{N,0})\right)\geq \delta,
\end{equation}
and using \eqref{Bvar} and \eqref{CSB2n},
\begin{equation}\nonumber
\frac{\Var LC_n}{n}\geq 2p(1-p)\delta^2.
\end{equation}

\subsection{On the order of the variance when one letter is omitted}

Let the letters $X_1,\dots,X_n$ be drawn from an alphabet $\alpha_1,\dots,\alpha_{m+1}$ and 
the letters $Y_1,\dots,Y_n$ drawn from an alphabet $\alpha_1,\dots,\alpha_{m}$: so $\alpha_{m+1}$ is an omitted letter, not belonging to 
any longest common subsequence.  (Let $m>1$, as the case $m=1$ is trivial and may be dealt with separately.)   Let $p:=\pr(X_i=\alpha_{m+1})>0$, and  in contrast to the binary-ternary case \cite{bonmat} or \cite{houdre2018variance}, no longer assume that the other letters are equiprobable and let 
$p_{X,1}:=\pr(X_i=\alpha_{1}),\dots,p_{X,m}:=\pr(X_i=\alpha_{m}),p_{Y,1}:=\pr(Y_i=\alpha_{1}),\dots,p_{Y,m}:=\pr(Y_i=\alpha_{m})$.  Since 

\begin{equation}\label{ineqvarb}
\Var LC_n \geq 2n B_{2n}(2n),
\end{equation}
it is therefore enough to find a constant lower bound on $B_{2n}(2n)$.  Firstly, by conditioning, 
\begin{align*}
B_{2n}(2n) & = \frac{1}{4n}\sum_{j=1}^{2n} \esp\left(\Delta_j LC_n (\Delta_j LC_n)^{1,\dots,j-1,j+1,\dots,n}\right)\\
&=  \frac{1}{4n}\sum_{j=1}^{2n}  \sum_{i,i'=1}^m \left( \esp \Delta_j LC_n^{Z_j=\alpha_i, Z'_j=\alpha_{i'}}\right)^2 \pr(Z_j=\alpha_i)\pr(Z'_j=\alpha_{i'}) \\
&\geq \frac{1}{4n}\sum_{j=1}^{n}  \sum_{i=1}^m \left( \esp \Delta_j LC_n^{X_j=\alpha_i, X'_j=\alpha_{m+1}}\right)^2 p_{X,i}p\\
&\geq \frac{1}{4n}\sum_{j=1}^{n} \left( \sum_{i=1}^m \esp \Delta_j LC_n^{X_j=\alpha_i, X'_j=\alpha_{m+1}} p_{X,i}\right)^2 p.
\end{align*}

\noindent
Letting $LC_{n-1,n}:=LCS(X_1 \cdots X_{n-1};Y_1\cdots Y_n)$, we have for any $j\in\ens{n}$,
\begin{equation}\nonumber
\sum_{i=1}^m  \esp \Delta_j LC_n^{X_j=\alpha_i, X'_j=\alpha_{m+1}} p_{X,i}=\esp(LC_n)-\esp(LC_{n-1,n}),
\end{equation}
and so,  
\begin{equation}\label{ineqb2n}
B_{2n}(2n) \geq \frac{1}{4}\left(\esp(LC_n)-\esp(LC_{n-1,n})\right) p.
\end{equation}

Let $(\pi,\eta)$ be the alignment of $(X_1,\dots,X_{n-1}),(Y_1,\dots,Y_n)$ which is minimal for the lexicographic order, so $(\pi,\eta)$ is well defined as a (measurable) function of $X_1,\dots,X_{n-1},Y_1,\dots,Y_n$. Let $F_n$ be the event "$\eta_{LC_n}<n$", in other words, $Y_n$ does not contribute to the longest common subsequences, then $\sum_{i=1}^n \Delta_n LC_n^{X_n=\alpha_i, X'_n=\alpha_{m+1}}\geq \mathds{1}_{F_n}$, hence
\begin{equation}\label{ineq1}
\esp(LC_n)-\esp(LC_{n-1,n})\geq p_{X,\min} \pr(F_n),
\end{equation}
where $p_{X,\min}:=\min_{1\leq i\leq m}p_{X,i}$.

We now combine this bound with some elements already present in \cite{houdre2018variance} (with its notations).    Let $V_1=\pi_1-1, V_2=\pi_2-\pi_1-1,\dots, V_{LC_n}=\pi_{LC_n}-\pi_{LC_n-1}-1$, and let $M$ be the number of indices $i$ such that $V_i>0$. In terms of \cite{houdre2018variance}, $M$ is the number of nonempty matches (except that there is also the term $V_1$). We denote by $I_{i,j}$ the event: "inserting $\alpha_i$ at the $j$-th position in $(X_1,\dots,X_{n-1}),(Y_1,\dots,Y_n)$ increases the longest common subsequence". Observe that
\begin{align}\label{ineq2}
\esp(LC_n)-\esp(LC_{n-1,n})& =\esp \left( LCS(X_1 \cdots X_{j-1} X'_1 X_{j} \cdots X_{n-1};Y_1 \cdots,Y_n)- LCS(X_1\cdots X_{n-1};Y_1\cdots Y_n)\right) \nonumber\\
 & =\sum_{i=1}^m p_{X,i}\pr(I_{i,j})\nonumber\\
 &=\frac{1}{n} \sum_{j=1}^n \sum_{i=1}^m p_{X,i}\pr(I_{i,j})\nonumber\\
 &\geq \frac{p_{X,\min}}{n} \esp \sum_{j=1}^n \sum_{i=1}^m I_{i,j}\nonumber\\
 &\geq p_{X,\min} \frac{\esp M}{n}.
\end{align}
From \eqref{ineq1} and \eqref{ineq2}, we get
\begin{equation}\label{ineqdeltae}
\esp(LC_n)-\esp(LC_{n-1,n})\geq \frac{p_{X,\min}}{2}\left(\pr(F_n)+\frac{\esp M}{n}\right).
\end{equation}
Let $\gamma^{*}$ be the limit of $\esp(LC_n)/n$, we have $\gamma^{*}\leq 1-p<1$, i.e., $0< p < 1-\gamma^*$.  Fix $k_0>0$ such that
\begin{equation}\nonumber
\sum_{k>k_0} mk(1-p_{Y,\min})^k\leq \frac{1-\gamma^{*}}{2}.
\end{equation}
When $F_n$ does not hold, that is, when $\pi_n=LC_n$, we have
\begin{equation}\nonumber
\sum_{i=1}^{LC_n}V_i=n-LC_n,
\end{equation}
so
\begin{equation}\nonumber
\esp \left(\sum_{i=1}^{LC_n}V_i\right) \geq \esp \left((n-LC_n)\mathds{1}_{F_n^c}\right)\geq \esp \left(n-LC_n\right)-\pr(F_n)n\geq (1-\gamma^{*})n - \pr(F_n)n.
\end{equation}

\noindent
Moreover, 
\begin{equation}\nonumber
k_0\esp M \geq \esp \left(\sum_{i=1}^{LC_n}V_i\mathds{1}_{V_i\leq k_0}\right).
\end{equation}
On the other hand, $(\pi,\eta)$ is minimal, so any unmatched gap has (at least) a letter of the alphabet 
which is not used, namely, the letter used in the next match. 
Therefore the average number of indices $i$ such that $V_i=k$ is no more than $nm(1-p_{Y,\min})^k$, and 
\begin{equation}\nonumber
\esp \left(\sum_{i=1}^{LC_n}V_i\mathds{1}_{V_i>k_0}\right)\leq n \sum_{k>k_0} mk(1-p_{Y,\min})^k \leq \frac{1-\gamma^{*}}{2} n.
\end{equation}
Finally we get 
\begin{equation}\nonumber
k_0 \esp M \geq \frac{1-\gamma^{*}}{2}n-\pr(F_n)n,
\end{equation}
and
\begin{equation}\nonumber
\pr(F_n)+\frac{\esp M}{n}\geq \frac{k_0 \esp M + \pr(F_n)n}{k_0 n}\geq \frac{1-\gamma^{*}}{2k_0},
\end{equation}
so putting it together with \eqref{ineqvarb}, \eqref{ineqb2n} and \eqref{ineqdeltae}, we get 
\begin{equation}\nonumber
\Var LC_n \geq \frac{p p_{X,\min}(1-\gamma^{*})}{8k_0}n.
\end{equation}

\subsection{A weaker kind of lower bound}

Let us return to the Bernoulli framework with parameter $0<p<1$, and let $\gamma_n(p)=\esp LC_n/n$ and $\gamma(p)=\lim_{n\to\infty}\gamma_n(p)$. It seems reasonable to expect that $\Var LC_n /n$ converges when $n$ tends to infinity, but, so far, unfortunately  a proof of this result has been elusive.  To the best of our knowledge, it is still an open problem to determine whether or not, in the uniform case, the variance tends to infinity. The function $\gamma$ is clearly symmetric around $1/2$, and it is expected to be strictly convex with a minimum at $1/2$, but besides numerical simulations there is no proof of this fact yet. The goal of this section is to prove:
\begin{thm}\label{thm:varsup}
Let $p_0\in (0,1/2)$ be such that $\gamma(p_0)>\gamma(1/2)$. Then there exists $p_1\in (p_0,1/2)$ such that when $p=p_1$, $$\limsup_{n\to\infty} \frac{\Var LC_n}{n} \geq 2p_0(1-p_0)\left(\frac{\gamma(p_0)-\gamma(1/2)}{1/2-p_0}\right)^2.$$
\end{thm}

\begin{rem}Using the bound $\gamma(1/2)<0.8263$ from \cite{lueker2009improved}, and since $\gamma(p)\geq p^2+(1-p)^2$, we can  apply the above 
theorem with $p_0=0.096$, to get for some $p_1\in (0.096,0.5), \limsup_{n\to\infty} \Var LC_n/n\geq 1.8/10^{8}$. Clearly, by symmetry, this limsup result is also valid for some $p_2=1-p_1\in (0.5, 0.904)$.
\end{rem}

\begin{proof}
We have
\begin{align*}
\gamma_n(p_0)-\gamma_n(1/2)  =-\int_{p_0}^{1/2} \frac{\mathrm{d}\gamma_n}{\mathrm{d}p}(p)\mathrm{d}p
=\int_{p_0}^{1/2} \frac{1}{2n}\sum_{k=1}^{2n} \esp_p\left(LC_n^{k,0}-LC_n^{k,1}\right)\mathrm{d}p,\\
\end{align*}
using a Russo-Margulis type formula. This is not strictly the Russo-Margulis lemma since  $LC_n$ is not monotone, but the proof of this version is 
elementary: as in \cite{garban2014noise}, we rewrite $\gamma_n$ as a function of $2n$ parameters, the parameters of each letter (Bernoulli random variables):
\begin{align*}
\frac{\mathrm{d}\gamma_n}{\mathrm{d}p}(p) = \frac{\mathrm{d}\gamma_n}{\mathrm{d}p}(p,p,\dots,p)
=\sum_{k=1}^{2n} \frac{\mathrm{d}\gamma_n}{\mathrm{d}p_k}(p,p,\dots,p), 
\end{align*}
which yields the result.  Hence, 
\begin{align*}
\gamma(p_0)-\gamma(1/2) & =\limsup_{n\to \infty}\gamma_n(p_0)-\gamma_n(1/2)\\
& \leq \int_{p_0}^{1/2}\limsup_{n\to \infty} \frac{1}{2n}\sum_{k=1}^{2n} \esp_p\left(LC_n^{k,0}-LC_n^{k,1}\right)\mathrm{d}p,
\end{align*}
so there exists $p_1\in (p_0,1/2)$ such that 

\begin{equation*}
\limsup_{n\to \infty} \frac{1}{2n}\sum_{k=1}^{2n} \esp_{p_1}\left(LC_n^{k,0}-LC_n^{k,1}\right)\geq \frac{\gamma(p_0)-\gamma(1/2)}{1/2-p_0}.
\end{equation*}

\noindent
Let us fix $p=p_1$. As seen previously, 
\begin{align*}
B_{2n}(2n) & =\frac{1}{2n}\sum_{k=1}^{2n}\left(\esp\left(LC_n(Z^{k,0})-LC_n(Z^{k,1})\right)\right)^2p_1(1-p_1),
\end{align*}
so
\begin{align*}
\Var LC_n & \geq \sum_{k=1}^{2n}\left(\esp\left(LC_n(Z^{k,0})-LC_n(Z^{k,1})\right)\right)^2p_0(1-p_0)\\
& \geq  2n \left(\frac{1}{2n}\sum_{k=1}^{2n} \esp_{p_1}\left(LC_n^{k,0}-LC_n^{k,1}\right)\right)^2p_0(1-p_0),
\end{align*}
and finally 
\begin{align*}
\limsup_{n \to \infty} \frac{\Var LC_n}{n} & \geq 2p_0(1-p_0)\left(\frac{\gamma(p_0)-\gamma(1/2)}{1/2-p_0}\right)^2.
\end{align*}
\end{proof}

\begin{rem}
	As already mentioned, it is expected that the function $\gamma$ is strictly convex, but even 
	proving that $\gamma$ is non-increasing on $[0,1/2]$ and non-decreasing on $[1/2, 1]$ seems to be lacking.  It also seems reasonable that for a fixed alphabet, say binary uniform, the sequence $(\esp LC_n/n)_{n\ge 1}$ is non-decreasing, but again a proof is lacking.  
\end{rem}

\subsection{On the order of the variance in the binary uniform case} 

A long-standing open problem is to find the order of the variance of $LC_n$ when the distribution is uniform. In this section, we focus on the uniform binary case, so $\lim_{n\to \infty} \esp LC_n/n=\gamma(1/2):=\gamma_2$.  We recall, from \cite{houdre2016closeness}, the definition of the function $\widetilde{\gamma}$: for any $p>0$,
\begin{equation}
\widetilde{\gamma}(p):=\lim_{n\to\infty} \frac{\esp \left(LCS(X_1\cdots X_n;Y_1\cdots Y_ {\lfloor np \rfloor})\right)}{n(1+p)/2}.
\end{equation}
By a superadditivity argument, this limit is well defined and $\widetilde{\gamma}$ is concave, non-decreasing on $[0,1]$ and non-increasing on $[1,+\infty)$ (for the details, and further properties of $\widetilde{\gamma}$, we refer the reader to \cite{houdre2016closeness}).

By symmetry, in this case, $B_{2n}(2n)=0$. However, letting $Z_1=(X_1,Y_1), Z_2=(X_2,Y_2),\dots, Z_n=(X_n,Y_n)$, then we may see the last $B_k$ is $B_n(n)$ ($LC_n=LC(Z_1,\dots,Z_n)$), which can be written as, by conditioning,
\begin{align*}
B_{n}(n) & = \frac{1}{2n}\sum_{j=1}^{n} \esp\left(\Delta_j LC_n (\Delta_j LC_n)^{1,\dots,j-1,j+1,\dots,n}\right)\\
&=  \frac{1}{2n}\sum_{j=1}^{n}  \sum_{\varepsilon,\varepsilon'\in\{0,1\}^2}\left( \esp \Delta_j LC_n^{Z_j=\varepsilon, Z'_j=\varepsilon'}\right)^2 \pr(Z_j=\varepsilon)\pr(Z'_j=\varepsilon') \\
&=\frac{1}{2n}\sum_{j=1}^{n} 8\left( \esp \Delta_j LC_n^{Z_j=(0,0), Z'_j=(0,1)}\right)^2 \pr(Z_j=(0,0))\pr(Z'_j=(0,1),
\end{align*}
where we used symmetry and the fact that when $\varepsilon'=\varepsilon$ or $\varepsilon'=(1,1)-\varepsilon$, $\esp \Delta_j LC_n^{Z_j=\varepsilon, Z'_j=\varepsilon'}=0$. Finally,
\begin{align*}
B_{n}(n) & = \frac{1}{4n}\sum_{j=1}^{n} \left( \esp \Delta_j LC_n^{Z_j=(0,0), Z'_j=(0,1)}\right)^2
\end{align*}
which may also be written
\begin{equation}
B_n(n)=\frac{1}{n} \sum_{j=1}^n \frac{1}{4}\left(\esp\left(LCS(Z^{j,(0,0)})-LCS(Z^{j,(0,1)})\right)\right)^2.
\end{equation}
So it is enough to find a lower bound on this quantity, which is doable for the terms on the edge ($1$ or $n$) but seems tricky for the terms in the middle. 

We may also fix $b\geq 2$ and let $Z_1=X_1,\dots,X_b$, $Z_2=X_{b+1},\dots,X_{2b}$, $\dots$. In this case, one gets that lower bounding $B_n(n)$ amounts to finding $w_1,w_2\in\{0,1\}^b$ and $\delta>0$ such that for all $n\geq 1$,
\begin{equation}
\frac{1}{n} \sum_{j=1}^n \left(\esp\left(LCS(Z^{j,w_1})-LCS(Z^{j,w_2})\right)\right)^2\geq \delta.
\end{equation}

For example, intuitively, it is likely to get a larger LCS with $w_1=(1,0)$ than with $w_2=(1,1)$, and with $w_1=(1,0,1,0,1)$ than with $w_2=(1,1,1,1,1)$. Running simulations in Python, Figure \ref{fig:plotsimul} seems to indicate that $B_n(n)$ is lower bounded by a strictly positive constant (which would yield the linearity of the variance).

\begin{figure}[h!]
\includegraphics[scale=0.6]{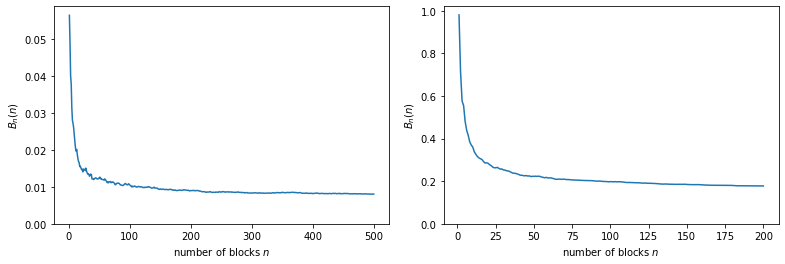}
\caption{$B_n(n)$ for $w_1=(1,0), w_2=(1,1)$ (left), and $w_1=(1,0,1,0,1)$, $w_2=(1,1,1,1,1)$ (right), with the empirical measure over $1000$ simulations.}
\label{fig:plotsimul}
\end{figure} 

We now pick again $Z_1=X_1,Z_2=X_2,\dots,Z_{2n}=Y_n$, and study $B_1(2n)$. Note that if $B_1(2n)$ was converging to zero, this would rule out the possibility of a linear lower bound on the variance. In the following, we study $B_1(2n)$, and find that it is lower bounded by a constant.

Let $X_1,\dots,X_n,Y_1,\dots,Y_n$ be independent Bernoulli random variables with parameter $1/2$, and let $\upsilon\leq n$. We may assume, for ease of notations, that $n=\upsilon m$ is a multiple of $\upsilon$, but it is not hard to adapt all the following proofs to the general case.  Let  $\mathcal{R}:=\{\overrightarrow{r}\in{\mathbb{N}}^m: 1=r_0\leq r_1\leq \dots\leq r_m=n\}$, and, for any $\overrightarrow{r}\in \mathcal{R}$, let
\begin{equation}
LC_n(\overrightarrow{r})=\sum_{i=0}^{m-1} LCS\left(P_i\right),
\end{equation} 
where $P_i:=\left((X_{\upsilon i+1},\dots,X_{(\upsilon+1)i}),(Y_{r_i},\dots,Y_{r_{i+1}-1})\right)$ (with the convention $(Y_{r_i},\dots,Y_{r_{i+1}-1})=\emptyset$ if $r_{i}=r_{i+1}$). For any $\overrightarrow{r}\in \mathcal{R}$, call $\overrightarrow{r}$ an alignment if $LC_n=LC_n(\overrightarrow{r})$.

Denote by $N_i$ the number of letters in the cell $P_i$, that is $\upsilon+r_{i+1}-r_i$. For any $\overrightarrow{r}\in\mathcal{R}$, let $I_{p_1,p_2}(\overrightarrow{r})=\{i\in\{0,\dots,m-1\}; r_{i+1}-r_i\in [\upsilon p_1,\upsilon p_2]\}$, and $(I_{p_1,p_2}(\overrightarrow{r}))^c$ its complement in $\{0,\dots,m-1\}$.  Next, let $B^n_{\varepsilon,p_1,p_2}$ be the event that: for any alignment $\overrightarrow{r}$, $$\sum_{i \in I_{p_1,p_2}(\overrightarrow{r})}N_i\geq \left(1-\frac{\varepsilon}{2}\right)2n.$$ Note that, recalling the notation $A^n_{\varepsilon,p_1,p_2}$ as defined in \cite{houdre2016closeness}, we have $B^n_{\varepsilon,p_1,p_2}\subset A^n_{\varepsilon,p_1,p_2}$. Indeed, if $A^n_{\varepsilon,p_1,p_2}$ is not satisfied, there is an alignment $\overrightarrow{r}$ (in \cite{houdre2016closeness}, the definition of an alignment is with strict inequalities rather than our non-strict inequalities, therefore an alignment as defined in \cite{houdre2016closeness} is necessarily also an alignment as defined here) such that the cardinality of $I_{p_1,p_2}(\overrightarrow{r})$ is strictly greater than $\varepsilon m$, which implies
\begin{align*}
\sum_{i\in (I_{p_1,p_2}(\overrightarrow{r}))^c}N_i & > \nu\varepsilon m\\
&>\varepsilon n,
\end{align*}
hence $\sum_{i \in I_{p_1,p_2}(\overrightarrow{r})}N_i< 2n-\varepsilon n$, indicating that  $B^n_{\varepsilon,p_1,p_2}$ is not satisfied. Hence $B^n_{\varepsilon,p_1,p_2}\subset A^n_{\varepsilon,p_1,p_2}$, and $\pr\left(B^n_{\varepsilon,p_1,p_2}\right)\leq \pr\left(A^n_{\varepsilon,p_1,p_2}\right)$. Therefore, the following is a strengthening of \cite[Theorem 2.2]{houdre2016closeness} (which is the same statement but with $A^n_{\varepsilon,p_1,p_2}$ in place of $B^n_{\varepsilon,p_1,p_2}$).
\begin{lem}
Let $\varepsilon > 0$. Let $0 < p_1 < 1 < p_2$ be such that $\widetilde{\gamma}(p_1) < \widetilde{\gamma}(1)= {\gamma_2}$ and $\widetilde{\gamma}(p_2)<{\gamma_2}$ and let $\delta \in \left(0, \min({\gamma_2} - \widetilde{\gamma}(p_1), {\gamma_2}- \widetilde{\gamma}(p_2))\right)$.   
Fix the integer $\upsilon $ to be such that $(1 + \ln (1+\upsilon ))/\upsilon  \leq \delta^2 \varepsilon^2 /16$, then
\begin{equation}\nonumber
\pr\left(B^n_{\varepsilon,p_1,p_2}\right)\geq 1-\exp\left(-n\left(\frac{\delta^2 \varepsilon^2 }{16}-\frac{1 + \ln (1+\upsilon )}{\upsilon }\right)\right),
\end{equation}
for all $n$ large enough.
\end{lem}

\begin{proof}
Let $\overrightarrow{r}\in \mathcal{R}$ be such that $\sum_{i\in (I_{p_1,p_2}(\overrightarrow{r}))^c}N_i> \varepsilon n$. We first prove that 
\begin{equation}\nonumber
\esp\left(LC_n(\overrightarrow{r})-LC_n\right)\leq -\frac{\delta\varepsilon n}{2}, 
\end{equation}
for all $n$ large enough.
We follow the proof of \cite[Lemma 3.1]{houdre2016closeness}. Let $\delta^*=\min({\gamma_2} - \widetilde{\gamma}(p_1), {\gamma_2}- \widetilde{\gamma}(p_2))$. Using the superadditivity of $\widetilde{\gamma}$, we get
\begin{align*}
\esp\left(LC_n(\overrightarrow{r})\right) & \leq {\gamma_2} \left(\sum_{i\in I_{p_1,p_2}(\overrightarrow{r})}\frac{N_i}{2}\right) + ({\gamma_2}-\delta^*)\left(\sum_{i\in (I_{p_1,p_2}(\overrightarrow{r}))^c}\frac{N_i}{2}\right)\\
&\leq \left({\gamma_2} - \frac{\delta^* \varepsilon}{2}\right) n.
\end{align*}
Moreover, for $n$ large enough,
$$-\esp\left(LC_n\right)\leq -\left({\gamma_2}-\frac{(\delta^*-\delta)\varepsilon}{2}\right)n,$$
so combining together these two inequalities, we get the desired result: 
\begin{equation}\nonumber 
\esp\left(LC_n(\overrightarrow{r})-LC_n\right)\leq -\frac{\delta\varepsilon n}{2}.
\end{equation}
The end of the proof is exactly as in \cite{houdre2016closeness}, the only difference is, as pointed out in \cite[Remark 2.2]{houdre2023central}, that the cardinality of $\mathcal{R}$ is now ${n+\upsilon  \choose \upsilon }$ instead of ${n \choose \upsilon }$ so $\ln \upsilon $ becomes $\ln(1+\upsilon )$.

\end{proof}

\begin{thm}\label{thm:b1}
There exists $C>0$ such that for all $n$ large enough, $B_1(2n)\geq C$.
\end{thm}

\begin{proof}
For any $\overrightarrow{r}\in\mathcal{R}$, let $S(\overrightarrow{r})=\{i\in \{0,\dots,m-1\}; LCS(P_i)=\min(\upsilon ,r_{i+1}-r_i)\}$ the set of the indices of "saturated" cells, meaning that $LCS(P_i)$ is maximal given the size of the cell. We first show that for some $\varepsilon>0$, with high probability, for any alignment $\overrightarrow{r}$,  $|S(\overrightarrow{r})|\leq (1-\varepsilon)m$ (still using the notation $|.|=\text{Card}(.)$). The idea behind this fact is that the $\varepsilon m$ non-saturated cells will guarantee the lower bound on $B_1(2n)$, as changing their coordinates might increase $LC_n$.
Let $x=0.28, p_1=1-x, p_2=1/p_1$, we know from \cite{houdre2016closeness} that $\widetilde{\gamma}(p_1) < {\gamma_2}$ and $\widetilde{\gamma}(p_2) < {\gamma_2}$. Let $\eta=\frac{2(1-x)}{2-x}-{\gamma_2}$, from the upper bound ${\gamma_2}\leq  0.8263$, see \cite{lueker2009improved}, it that $\eta>0$.  Let $\varepsilon\in \left(0,\frac{\eta}{2({\gamma_2}+\eta)}\right)$, and, lastly, let $\delta \in \left(0, \min({\gamma_2} - \widetilde{\gamma}(p_1), {\gamma_2}- \widetilde{\gamma}(p_2))\right)$ and fix $\upsilon $ to be such that $(1 + \ln (1+\upsilon ))/\upsilon  < \delta^2 \varepsilon^2 /16$.

Let $C^n_\varepsilon$ be the event: for any alignment $\overrightarrow{r}$, $|S(\overrightarrow{r})|\leq(1-\varepsilon)m$. If $(C^n_\varepsilon)^c\cap B^n_{\varepsilon,p_1,p_2}$ is realized, then there is some alignment $\overrightarrow{r}$ such that $|S(\overrightarrow{r})|>(1-\varepsilon)m$, and
\begin{equation}
LC_n\geq \sum_{i\in S(\overrightarrow{r})\cap I_{p_1,p_2}(\overrightarrow{r})} \frac{N_i}{2}\frac{\min(\upsilon ,r_{i+1}-r_i)}{\frac{N_i}{2}}.
\end{equation}
For any $i\in I_{p_1,p_2}(\overrightarrow{r})$, $r_{i+1}-r_i\in [\upsilon p_1,\upsilon p_2]$ so $$\frac{\min(\upsilon ,r_{i+1}-r_i)}{\frac{N_i}{2}}\geq \frac{2}{1+p_2}=\frac{2p_1}{1+p_1}=\frac{2(1-x)}{2-x}={\gamma_2}+\eta,$$
so
\begin{equation}
LC_n\geq \sum_{i\in S(\overrightarrow{r})\cap I_{p_1,p_2}(\overrightarrow{r})} \frac{N_i}{2}({\gamma_2}+\eta).
\end{equation}
Moreover,
\begin{align*}
\sum_{i\in (S(\overrightarrow{r})\cap I_{p_1,p_2}(\overrightarrow{r}))^c} \frac{N_i}{2} & = \sum_{i\in (S(\overrightarrow{r}))^c\cap I_{p_1,p_2}(\overrightarrow{r})} \frac{N_i}{2}+\sum_{i\in (I_{p_1,p_2}(\overrightarrow{r}))^c}\frac{N_i}{2}\\
&\leq \upsilon  \frac{1+p_2}{2} \varepsilon m+\frac{\varepsilon n}{2}\\
&\leq 2\varepsilon n,
\end{align*} 
and therefore, 
\begin{equation}
LC_n\geq (1-2\varepsilon)({\gamma_2}+\eta)n.
\end{equation}
Given the choice of $\varepsilon$, $(1-2\varepsilon)({\gamma_2}+\eta)>{\gamma_2}$, so by concentration, this has probability exponentially small to happen.
Therefore, $\pr\left((C^n_\varepsilon)^c\right)\leq \pr\left((C^n_\varepsilon)^c\cap B^n_{\varepsilon,p_1,p_2}\right)+\pr \left((B^n_{\varepsilon,p_1,p_2})^c\right)$ goes to zero (exponentially fast) as $n$ goes to infinity.

\noindent
Now for $i\in\ens{m}$, let 
\begin{equation}
V_i=\max_{x\in\{0,1\}^\upsilon } \left| LCS(X_1\cdots X_{\upsilon (i-1)} x_1 \cdots x_\upsilon X_{\upsilon i+1}\cdots X_n ; Y_1 \cdots Y_n)-LCS(X_1 \cdots X_n; Y_1\dots Y_n)\right|.
\end{equation}

But, for any $i\in (S(\overrightarrow{r}))^c$,  it follows that $V_i\geq 1$, and so  
\begin{equation}
\esp\left(\frac{1}{m}\sum_{i=1}^m V_i^2\right)> \varepsilon \pr(C^n_\varepsilon).
\end{equation}
Now, for $x\in \{0,1\}^\upsilon $ and $j\in \{\upsilon (i-1)+1,\dots,\upsilon i\}$, let 
\begin{align*}
\delta_{j}(x)=& LCS(X_1 \cdots X_{\upsilon (i-1)} x_1\cdots x_{j-\upsilon (i-1)} X_{j+1} \cdots X_n; Y_1 \cdots Y_n)\\
& - LCS(X_1 \cdots X_{\upsilon (i-1)} x_1 \cdots x_{j-\upsilon (i-1)-1} X_{j} \cdots X_n; Y_1 \cdots Y_n),
\end{align*}
so that,
\begin{align*}
V_i^2 &=\max_{x\in\{0,1\}^\upsilon } \left|\sum_{j=\upsilon (i-1)+1}^{\upsilon i} \delta_j(x)\right|^2\leq \max_{x\in\{0,1\}^\upsilon } \upsilon \sum_{j=\upsilon (i-1)+1}^{\upsilon i} \delta_j(x)^2\leq \upsilon \sum_{j=\upsilon (i-1)+1}^{\upsilon i} \max_{x\in\{0,1\}^\upsilon } \delta_j(x)^2.
\end{align*}
Note that $\esp \Delta_j^2=\esp \delta_j(X'_1,\dots,X'_\upsilon )^2$ (where $\Delta_j$ is the difference in length between the original LCS and the LCS modified, via an independent copy of the variable 
at spot $j$, (see the next section for the precise definition of $\Delta_j$)), and
$$\esp_{X'_1,\dots,X'_\upsilon } \Delta_j^2\geq \frac{1}{2^\upsilon }\max_{x\in\{0,1\}^\upsilon } \delta_j(x)^2.$$
Hence, $\esp \Delta_j^2\geq \esp \max_{x\in\{0,1\}^\upsilon } \delta_j(x)^2/{2^\upsilon }$, and 
so $V_i^2  \leq \upsilon 2^\upsilon  \sum_{j=\upsilon (i-1)+1}^{\upsilon i}\esp \Delta_j^2$.  Therefore,  
\begin{align*}
\esp\left(\frac{1}{m}\sum_{i=1}^m V_i^2\right) \leq  \frac{\upsilon 2^\upsilon }{m}\sum_{j=1}^{n}\esp \Delta_j^2,
\end{align*}
finally,
\begin{align*}
\varepsilon \pr(C^n_\varepsilon)&< \upsilon ^2 2^\upsilon  B_1(2n),
\end{align*}
and thus, for $n$ large enough, $B_1(2n)>{\varepsilon}/{2\upsilon ^2 2^\upsilon }$.
\end{proof}

\begin{rem}
The above result is a necessary condition (certainly not sufficient, though) to have $\Var LC_n$ asymptotically linear.
This implies that there exists $C'>0$, such that for all $n$, $B_1(2n)\geq C'$, as for all $n$, $B_1(2n)>0$.
\end{rem}

\subsection{A note on a potential implication of \cite{houdre2023central}}

In this section, $\alpha\in (0,1)$, $\upsilon=n^\alpha$, and $\overrightarrow{r}$ is a random alignment.
Let $X'_1,\dots,X'_n, Y'_1,\dots,Y'_n$ be independent Bernoulli variables with parameter $1/2$, independent from all the previous variables. As previously, we write $Z=(Z_1,\dots,Z_{2n}) :=(X_1 ,\dots, X_n, Y_1, \dots, Y_n)$, and as in \cite{houdre2023central}, for $j\in\ens{2n}$, let
\begin{align*}
\Delta_j & := LCS(Z)-LCS(Z_1 \cdots Z'_j \cdots Z_{2n})\\
\widetilde{\Delta_j}& := LCS\left(P_i\right)-LCS\left(P'_i\right)
\end{align*}
where $P_i$ is the cell of length $\upsilon$  containing $Z_j$ and $P'_i$ is the same cell but with $Z'_j$ instead of $Z_j$.
We also write for $j,k\in \ens{m}$:
\begin{align*}
LC_n^j &:=LCS(Z_1 \cdots Z'_j \cdots Z_{2n})\\
LC_n^{j,k} &:=LCS(Z_1\cdots Z'_j \cdots Z'_k\cdots Z_{2n})\\
\Delta_{j,k}&:=LC_n-LC_n^j-LC_n^j+LC_n^{j,k}.
\end{align*}
It is claimed in \cite{houdre2023central} that $\esp\left|\widetilde{\Delta_j}-\Delta_j\right|=\esp\left(\widetilde{\Delta_j}-\Delta_j\right)$ is exponentially small in $n$. The equality comes from the fact that $\widetilde{\Delta_j}-\Delta_j\geq 0$ (as explained in \cite{houdre2023central}). Furthermore, $\esp \Delta_j=0$, so the problem boils to controlling $\esp \widetilde{\Delta_j}$. Let us assume, in this section, that $\esp \widetilde{\Delta_j}\leq \exp(-tn)$ for some $t>0$ not depending on $j, n$, and let us denote by $A_j$ the event $\widetilde{\Delta_j}-\Delta_j=0$. Of course, $\pr(A_j^c)\leq \exp(-tn)$. Finally, let $C_{j,k}$ be the event "$Z_j$ and $Z_k$ are not in the same cell". Let $j,k\in\ens{n}$ and suppose $A_j$, $A_k$ and $C_{j,k}$ are all  realized, then when $X_j$ is flipped to $X'_j$, the alignment $\overrightarrow{r}=\overrightarrow{r}(Z)$ is still an alignment for $(Z_1,\dots,Z'_j,\dots,Z_{2n})$, so
\begin{equation}
LC_n^{j,k}-LC_n^j\geq -\widetilde{\Delta_k}= LC_n^k-LC_n
\end{equation}
so, in other terms,
\begin{equation}\label{eq:bounddelt}
\Delta_{j,k}\mathds{1}_{A_j}\mathds{1}_{A_k}\mathds{1}_{C_{j,k}}\geq 0.
\end{equation}

\noindent
Let us write $\Delta_{j,k}=\Delta_{j,k}^+-\Delta_{j,k}^-$ (the positive and negative parts), using the bounds $|\Delta_{j,k}|\leq 2$ and \eqref{eq:bounddelt} we get 
\begin{equation}
\Delta_{j,k}^-\leq 2(1-\mathds{1}_{A_j}\mathds{1}_{A_k}\mathds{1}_{C_{j,k}})
\end{equation}
so $(\Delta_{j,k}^-)^2\leq 4(1-\mathds{1}_{A_j}\mathds{1}_{A_k}\mathds{1}_{C_{j,k}})$, and
\begin{equation}
\esp (\Delta_{j,k}^-)^2\leq 4\left(\pr(A_j^c)+\pr(A_k^c)+\pr(C_{j,k}^c)\right),
\end{equation}
\begin{equation}
\esp (\Delta_{j,k}^+)^2\leq 2\esp \Delta_{j,k}^+ =  2\esp \Delta_{j,k}^-\leq  4\left(\pr(A_j^c)+\pr(A_k^c)+\pr(C_{j,k}^c)\right),
\end{equation}
hence 
\begin{equation}
\esp (\Delta_{j,k})^2\leq 8\left(\pr(A_j^c)+\pr(A_k^c)+\pr(C_{j,k}^c)\right).
\end{equation}

We may now give an upper bound on $B_1(2n)-B_2(2n)$:
\begin{align*}
B_1(2n)-B_2(2n) &= \frac{1}{4(2n)(2n-1)} \sum_{\substack{j\neq k\\j,k\in\ens{2n}}} \esp (\Delta_{j,k})^2\\
&= \frac{2}{4(2n)(2n-1)} \sum_{\substack{j\neq k\\j\in\ens{n},k\in\ens{2n}}} \esp (\Delta_{j,k})^2 \quad\text{(by symmetry)}\\
&\leq  \frac{2}{n(2n-1)}\esp \left(\sum_{\substack{j\neq k\\j\in\ens{n}}} \mathds{1}_{C_{j,k}^c}\right)+ \frac{2}{n(2n-1)}\sum_{\substack{j\neq k\\j\in\ens{n}}}\left(\pr(A_j^c)+\pr(A_k^c)\right)\\
&\leq  \frac{2}{n(2n-1)} (2n\upsilon-n)+ 2\exp(-tn).
\end{align*}
So when $n$ is large enough,
\begin{equation}
B_1(2n)-B_2(2n)\leq \frac{2\upsilon}{n}, 
\end{equation}
and from the convexity of $B$, and using the lower bound $0<C\leq B_1(2n)$ (see Theorem \ref{thm:b1}),
\begin{equation*}
\Var LC_n=B_1(2n)+\dots+B_{2n}(2n)\geq \sum_{i=1}^{\frac{Cn}{2\upsilon}} C-\frac{2\upsilon(i-1)}{n}, 
\end{equation*}
which is equivalent to $C^2n/(4\upsilon)$. So for some constant $C'>0$,
\begin{equation}
\Var LC_n\geq C'n^{1-\alpha}.
\end{equation}

Once again, this is under the assumption that $\esp \widetilde{\Delta_j}\leq \exp(-tn)$. If, additionally, this assumption holds for some $\alpha<1/10$, then by \cite{houdre2023central} there is convergence of the properly rescaled $LC_n$ to a Gaussian.

There is also a somewhat weaker assumption that would guarantee the linearity of the variance. 
Recalling the percolation interpretation of the LCS, we denote by $\text{Geo}$ the (random) set of geodesics, and for any $a,b\in\ens{2n}$, $\text{Geo}^a$ the set of geodesics when the $Z_a$ is turned into $Z'_a$, and $\text{Geo}^{a,b}$ the set of geodesics when $Z_a$ is turned into $Z'_a$ and $Z_b$ is turned into $Z'_b$.  For $j,k\in\ens{m}$, let $A_{j,k}$ be the event: there exists $(p,q)$ such that $j<p<k$ and there exist $(g_1,g_2,g_3,g_4)\in \text{Geo} \cap \text{Geo}^j \cap \text{Geo}^k \cap \text{Geo}^{j,k}$ such that $(p,q)\in g_1\cap g_2 \cap g_3 \cap g_4$. In words, this is the event that it is possible to find $X_p$ aligned with $Y_q$ no matter the values of $X_j$ and $X_k$. Similarly, let $B_{j,k}$ be the event: there exists $(p,q)$ such that $j<p$ and $k>q$ or $j>p$ and $k<q$ and there exists  $(g_1,g_2,g_3,g_4)\in \text{Geo} \cap \text{Geo}^j \cap \text{Geo}^{k+n} \cap \text{Geo}^{j,k+n}$ such that $(p,q)\in g_1\cap g_2 \cap g_3 \cap g_4$. In words, this is the event that it is possible to find $X_p$ aligned with $Y_q$ no matter the values of $X_j$ and $X_k$, and such that $X_j, Y_k$ are not both "on the same side". Now suppose that $\pr(A_{j,k}^c), \pr(B_{j,k}^c)\leq \exp(\alpha |k-j|)$ for some constant $\alpha>0$. Then an adaptation of the proof above shows that the variance is lower bounded by $C'n$ for some constant $C'>0$.

\noindent
{\bf Acknowledgements:}  Many thanks to the ICTS in Bengaluru as well as the GESDA program at IHP in Paris for their hospitalities and support while part of this research was carried out and to R.~van Handel for his comments on Theorem 2.1, and to Alexandros Eskenazis for pointing out to us the references \cite{eskenazis2024bias} and \cite{ivanisvilistone}.

\bibliographystyle{plain}

\bibliography{biblio-6}

\end{document}